\UseRawInputEncoding
\documentclass[12pt, reqno]{amsart}
\usepackage[paperheight=307mm, paperwidth=210mm, left=20mm, right=20mm]{geometry}
\usepackage{amssymb,latexsym,amsmath,amscd,amsfonts,amsthm,amsxtra}
\usepackage{latexsym}
\usepackage[mathscr]{eucal}
\usepackage{bm}
\usepackage{rotating}
\usepackage{pgfplots}
\usepackage{hyperref}
\usepackage{emptypage}
\usepackage{enumerate}
\usepackage{enumitem}

\newcommand{\C}{\mathbb{C}}

\newcommand{\Cn}{\mathbb{C}^n}
\newcommand{\D}{\mathbb{D}}
\newcommand{\T}{\mathbb{T}}

\newcommand{\gn}{\mathbb{G}_n}
\newcommand{\Gn}{\widetilde{\mathbb{G}}_n}
\newcommand{\gamn}{\Gamma_n}
\newcommand{\Gamn}{\widetilde{\Gamma}_n}

\newcommand{\be}{\beta}

\newcommand{\q}{\quad}
\newcommand{\qq}{\qquad}
\newcommand{\n}{\lVert}
\newcommand{\imp}{\Rightarrow}
\newcommand{\Lra}{\Leftrightarrow}
\newcommand{\lf}{\left(}
\newcommand{\ls}{\left\{}
\newcommand{\lt}{\left[}
\newcommand{\rf}{\right)}
\newcommand{\rs}{\right\}}
\newcommand{\rt}{\right]}

\newcommand{\nj}{{n \choose j}}

\def ~{\hspace{1mm}}

        \def\textmatrix#1&#2\\#3&#4\\{\bigl({#1 \atop #3}\ {#2 \atop #4}\bigr)}
        \def\dispmatrix#1&#2\\#3&#4\\{\left({#1 \atop #3}\ {#2 \atop #4}\right)}

\newtheorem{thm}{Theorem}[section]

\newtheorem{lemma}[thm]{Lemma}

\numberwithin{equation}{section} \theoremstyle{definition}
\newtheorem{defn}[thm]{Definition}
\newtheorem{rem}[thm]{Remark}

\begin{document}

%\pagenumbering{Roman} \thispagestyle{empty}
\title[Characterizations of the symmetrized polydisc]
{Characterizations of the symmetrized polydisc via another family of domains}

\author[Sourav Pal]{Sourav Pal}
\address[Sourav Pal]{Mathematics Department, Indian Institute of Technology Bombay, Powai, Mumbai - 400076, India.} 
\email{souravmaths@gmail.com , sourav@math.iitb.ac.in}

\author[Samriddho Roy]{Samriddho Roy}
\address[Samriddho Roy]{Mathematics Department, Indian Institute of Technology Bombay, Powai, Mumbai - 400076, India.} \email{samriddhoroy@gmail.com}

\keywords{Symmetrized polydisc, extended symmetrized polydisc}

\subjclass[2010]{32A07, 32E20, 32Q02}

\thanks{The named author was supported by the Seed Grant of IIT Bombay, the CPDA of Govt. of India, the INSPIRE Faculty Award (Award No. DST/INSPIRE/04/2014/001462) of DST, India and the MATRICS Award of SERB, (Award No. MTR/2019/001010) of DST, India. The second named author was supported by a Ph.D fellowship from the University Grand Commission of India.}

\begin{abstract}

We find new characterizations for the points in the \textit{symmetrized polydisc} $\mathbb G_n$, a family of domains associated with the spectral interpolation, defined by
\[
\mathbb G_n :=\left\{ \left(\sum_{1\leq i\leq n} z_i,\sum_{1\leq
i<j\leq n}z_iz_j \dots, \prod_{i=1}^n z_i \right): \,|z_i|<1,
i=1,\dots,n \right \}.
\]
We introduce a new family of domains which we call \textit{the extended symmetrized polydisc} $\widetilde{\mathbb G}_n$, and define in the following way:
\begin{align*}
\widetilde{\mathbb G}_n := \Bigg\{ (y_1,\dots,y_{n-1}, q)\in
\mathbb C^n :\; q \in \mathbb D, \;  y_j = \beta_j + \bar{\beta}_{n-j} q, \; \beta_j \in \mathbb C &\text{ and }\\
|\beta_j|+ |\beta_{n-j}| < {n \choose j} &\text{ for } j=1,\dots, n-1 \Bigg\}.
\end{align*}
We show that $\mathbb G_n=\widetilde{\mathbb G}_n$ for $n=1,2$ and that ${\mathbb G}_n \subsetneq \widetilde{\mathbb G}_n$ for $n\geq 3$. We first obtain a variety of characterizations for the points in $\widetilde{\mathbb G}_n$ and we apply these necessary and sufficient conditions to produce an analogous set of characterizations for the points in ${\mathbb G}_n$. Also we obtain similar characterizations for the points in $\Gamma_n \setminus {\mathbb G}_n$, where $\Gamma_n =\overline{{\mathbb G}_n}$. A set of $n-1$ fractional linear transformations play central role in the entire program. We also show that for $n\geq 2$, $\widetilde{\mathbb G}_n$ is non-convex but polynomially convex and is starlike about the origin but not circled.

\end{abstract}

\maketitle

%\tableofcontents    % Give Table of Contents
%\markboth{Contents}{Contents}

\section{Introduction}

\vspace{0.3cm}

\noindent The following notations and terminologies will be used throughout the paper. We denote by $\mathbb{R}$ and $\mathbb{C}$ the sets of real
numbers and complex numbers respectively. By $\mathbb{D}$ and
$\mathbb{T}$ we mean the open unit disc and
the unit circle respectively with centre at the origin of $\mathbb{C}$. We follow the standard convention to denote by $\mathcal M_{m\times n}(\mathbb C)$ (or $\C^{m \times n}$) the space
of all $m \times n$ complex matrices and $\mathcal M_n(\mathbb C)$
is used when $m=n$. For a matrix $A \in \mathcal M_n(\mathbb C)$,
$\n A \n$ is the operator norm of $A$. Also, the spectrum and the spectral radius of a bounded operator $T$ are
denoted by $\sigma(T)$ and $r(T)$ respectively.\\

In this article, we contribute to the understanding of the geometry of the symmetrized polydisc $\gn$ by producing a set of necessary and sufficient conditions each of which characterizes a point in $\gn$, where $\gn$ is the following family of domains:
\[
\mathbb G_n :=\left\{ \left(\sum_{1\leq i\leq n} z_i,\sum_{1\leq
i<j\leq n}z_iz_j \dots, \prod_{i=1}^n z_i \right): \,|z_i|<1,
i=1,\dots,n \right \}.
\]
This family of domains $\{ \gn \}$ is naturally associated with the spectral Nevanlinna-Pick interpolation in the following way: a matrix $T$ belongs to the spectral unit ball $\mathcal B_n^1 \subset \mathcal M_n(\C)$, i.e., $r(T)<1$  if and only if
$\pi_n(\xi_1,\dots, \xi_n) \in \mathbb G_n$ (see
\cite{costara1}). Here $\xi_1, \dots , \xi_n$ are the
eigenvalues of $T$ and $\pi_n$ is the symmetrization map on
$\mathbb C^n$ defined by
\[
\pi_n(z_1,\dots, z_n) = \left(\sum_{1\leq i\leq n} z_i,\sum_{1\leq
i<j\leq n}z_iz_j,\dots, \prod_{i=1}^n z_i \right).
\]
For given distinct points $\eta_1,\dots, \eta_k$ in $\D$ and matrices $T_1,\dots , T_k \in \mathcal B_n^1$, the spectral Nevanlinna-Pick interpolation seeks necessary and sufficient conditions under which there exists an analytic function $F:\D \rightarrow \mathcal M_n(\C)$ that interpolates the data, that is, $F(\eta_i)=T_i$ for $i=1,\dots ,k$. Apart from the derogatory matrices, the $n\times n$ spectral Nevanlinna-Pick problem is equivalent to a similar interpolation problem for the symmetrized polydisc $\gn$ (see \cite{ay-ieot}, Theorem 2.1). Note that a bounded domain like $\mathbb G_n$, which has complex-dimension $n$, is much easier to deal with than an unbounded $n^2$-dimensional object like $\mathcal B_n^1$. The symmetrized polydisc has attracted considerable attention in past two decades because of its rich function theory \cite{ALY12, Bh-Sau, zwonek1, pal-roy 2, PZ, zwonek3}, elegent complex geometry \cite{ALY14, AY04, bharali, costara, edi-zwo, NN, zwonek4, zwonek5} and associated operator theory \cite{ay-jfa, ay-jot, tirtha-sourav, tirtha-sourav1, Bisai-Pal1, sourav, sourav3, pal-shalit} (also see references there in). It is evident from the definition that $\mathbb G_1=\D$ which is convex but for $n\geq 2$, $\gn$ is non-convex but polynomially convex (see \cite{ay-jfa, edi-zwo}).\\

The symmetrized polydisc $\gn$ and its closure $\gamn$ are the image of polydisc $\D^n$ and the closed polydisc $\overline{\mathbb D^n}$ respectively under the symmetrization map $\pi_n$. We obtain from the literature (see \cite{costara1, edi-zwo}) that
\[
\gamn=\overline{\gn} :=\left\{ \left(\sum_{1\leq i\leq n} z_i,\sum_{1\leq
i<j\leq n}z_iz_j \dots, \prod_{i=1}^n z_i \right): \,|z_i|\leq 1,
i=1,\dots,n \right \}.
\]
In \cite{costara1}, Costara describe the points in $\gn$ and $\gamn$ in the following way:
\[
\gn = \ls (s_1,\dots, s_{n-1}, p)\in \C^n : p\in \D,\; s_j=\be_j +
\bar\be_{n-j}p \: \text{ and } \: (\be_1, \dots, \be_{n-1}) \in
\mathbb G_{n-1} \rs .
\]
\[
\gamn=\overline{\gn} = \ls (s_1,\dots, s_{n-1}, p)\in \C^n : p\in \overline{\D},\; s_j=\be_j +
\bar\be_{n-j}p \: \text{ and } \: (\be_1, \dots, \be_{n-1}) \in
\Gamma_{n-1} \rs .
\]

Taking cue from Costara's characterization, we introduce a new family of domains, namely the \textit{extended symmetrized polydisc} $\Gn$, which we define in the following way:
\begin{align*}
 \widetilde{\mathbb G}_n : = \Bigg\{ (y_1,\dots,y_{n-1}, q)\in \C^n
:
|q|<1 , \:  y_j = \be_j + \bar \be_{n-j} q \: \text{ with }\;  |\beta_j| & +  |\beta_{n-j}| < {n \choose j},\\ & 1\leq j\leq n-1 \Bigg\}.
\end{align*}
It is obvious that if $(s_1,\dots, s_{n-1},p)\in \mathbb G_n$, then $|s_i|< {n \choose i}$. So, if $\lf \be_1,\dots, \be_{n-1} \rf \in \mathbb G_{n-1}$,
then $|\be_j| + |\be_{n-j}| < {n-1 \choose j}+{n-1 \choose n-j}= \nj $.  Therefore, it follows that $\gn \subseteq \Gn$. It is evident that $\widetilde{\mathbb G_1}={\mathbb G_1}=\D$ and we shall prove that $\widetilde{\mathbb G_2}={\mathbb G_2}$ and that $\mathbb G_n \subsetneq \widetilde{\mathbb G_n}$ for $n\geq 3$. We first characterize the points in $\Gn$ and its closure $\Gamn$ in several different ways. Then we use these characterizations to obtain analogous characterizations for the points in $\gn$ and its closure $\gamn$. To characterize the points in $\Gn$ and $\Gamn$, we introduce $n-1$ fractional linear transformations $\Phi_1,\dots ,\Phi_{n-1}$, which play pivotal role in this paper. Also we show that $\Gamn$ is neither convex nor circled but polynomially convex and starlike and hence is simply connected.\\

\noindent \textbf{Note.} This article is a part of authors' unpublished note \cite{pal-roy 3}. The techniques that are employed in the proofs of Theorems \ref{char G 3} and \ref{char F 3} are generalization in several variables of the techniques that are used in the article \cite{awy}.\\

\section{Characterizations of the points in $\Gn$ and $\Gamn$}

\vspace{0.3cm}

\noindent For characterizing the points in $\Gn$ and $\widetilde{\Gamma_n}$, we introduce $(n-1)$ fractional linear transformations $\Phi_1, \dots, \Phi_{n-1}$ in the following way.
\begin{defn}
For $z \in \C$, $y=(y_1,\dots,y_{n-1},q) \in \Cn $ and for any $j\in \left\{1,\dots,n-1\right\}$,
    let us define a meromorphic map
    \begin{equation} \label{defn-p}
        \Phi_j(z,y) =
        \begin{cases}
            \dfrac{{n \choose j}qz-y_j}{y_{n-j}z-{n \choose j}} & \q \text{ if } y_j y_{n-j}\neq {n \choose j}^2 q \,, \\
            \\
            \dfrac{y_j}{{n \choose j}} & \q \text{ if } y_j y_{n-j} = {n \choose j}^2 q \, ,
        \end{cases}
    \end{equation}
    \\ and
    \begin{equation} \label{defn-D}
        D_j(y) =\underset{z \in \D}{\sup}|\Phi_j (z,y)|= \n \Phi_j(.,y) \n _{H^\infty} \, ,
    \end{equation}\\
    where ${H^\infty}$ denotes the Banach space of bounded complex-valued analytic functions on $\D$ equipped with supremum norm.
\end{defn}

Needless to mention that for a fixed $y=(y_1,\dots, y_{n-1},q)\in \Cn$, the function $\Phi_j(.,y)$ is a M\"{o}bius transformation. It is evident from the definition that $\Phi_j(.,y) = \Phi_{n-j}(.,\tilde y)$,
where
\[\tilde y = (y_1,\dots, y_{j-1}, y_{n-j},y_{j+1}, \dots, y_{n-j-1}, y_j,y_{n-j+1}, \dots,y_{n-1},q) \in \C^n,
\]
which is obtained by interchanging the $i$-th and the $j$-th coordinates of $y$. It can be shown by a few steps of calculations that if $|y_{n-j}| \neq \nj$, then the image set $\Phi_j(\mathbb T, y)$ is a circle whose center and radius are
\begin{equation}\label{cent-radi formula}
    \dfrac{{n \choose j} \left(y_j - \bar y_{n-j} q \right)}{{n \choose j}^2 - |y_{n-j}|^2} \q \text{and} \q \dfrac{\left|y_j y_{n-j} - {n \choose j}^2 q \right|}{{n \choose j}^2 - |y_{n-j}|^2}
\end{equation}
respectively. If $|y_{n-j}|< {n \choose j}$, then $\Phi_j(.,y)$ is bounded on $\D$. By the maximum modulus principle, the set $\Phi_j(\D,y)$ must be lying inside the circle $\Phi_j(\mathbb{T},y)$ for that $j\in \{1 ,\dots,n-1\}$. Now by continuity of the function $\Phi_j(.,y)$, the set  $\Phi_j(\mathbb{D},y)$ is the open disc with boundary $\Phi_j(\mathbb{T},y)$. Hence for a point $y =(y_1,\dots, y_{n-1},q) \in \Cn$ and $j \in \{1 ,\dots,n-1\}$ if $|y_{n-j}|< {n \choose j}$, then the function $\Phi_j(.,y)$ maps $\D$ to the open disc with center and radius as in (\ref{cent-radi formula}). Note that the point of maximum modulus of a closed disc with center $a$ and radius $r$ has modulus $|a| +r$. Hence
\begin{equation} \label{formula-D}
    D_j(y)=\begin{cases}
        \dfrac{{n \choose j} \left|y_j - \bar y_{n-j} q \right| + \left|y_j y_{n-j} - {n \choose j}^2 q \right|}{{n \choose j}^2 - |y_{n-j}|^2 } & \q \text{ if } |y_{n-j}| < {n \choose j} \text{ and } y_j y_{n-j}\neq {n \choose j}^2 q \,, \\ \\
        \dfrac{|y_j|}{{n \choose j}} & \q \text{ if } y_j y_{n-j} = {n \choose j}^2 q \,, \\ \\
        \infty & \q \textrm{otherwise}\, .
    \end{cases}
\end{equation}

\subsection{Characterizations for the points in $\Gn$}\label{char Gn}
Being armed with the functions $\Phi_j$'s and $D_j$'s, we now characterize the points in $\Gn$. 

\begin{thm}\label{char G 3}
    For a point $y =(y_1,\dots, y_{n-1},q) \in \Cn$, the following are equivalent:
    \begin{enumerate}
        \item[$(1)$] $y \in \Gn$;\\
        \item[$(2)$] ${n \choose j} - y_j z - y_{n-j}w + { n \choose j} qzw \neq 0$, for all $z,w \in \overline\D$ and for all $j = 1, \dots, \left[\frac{n}{2}\right]$;\\
        \item[$(3)$] for all $j = 1, \dots, \left[\frac{n}{2}\right]$,
        \[
        \n \Phi_j(.,y)\n_{H^{\infty}} < 1 \; \text{and if } y_jy_{n-j}= {n \choose j}^2 q \; \text{ then,}
        \;\text{ in addition, }\; |y_{n-j}|< {n \choose j}
        \]
        \item[$(3')$] for all $j = 1, \dots, \left[\frac{n}{2}\right]$,
        \[
        \n \Phi_{n-j}(.,y)\n_{H^{\infty}} < 1 \; \text{and if } y_jy_{n-j}= {n \choose j}^2 q \; \text{ then,}
        \;\text{ in addition, }\; |y_j|< {n \choose j} ;
        \]
        \item[$(4)$] for all $j = 1, \dots, \left[\frac{n}{2}\right]$,
        \[
        {n \choose j}\left|y_j - \bar y_{n-j} q\right| + \left|y_j y_{n-j} - {n \choose j}^2 q \right| < {n \choose j}^2 -|y_{n-j}|^2
        \]
        \item[$(4')$] for all $j = 1, \dots, \left[\frac{n}{2}\right]$,
        \[
        {n \choose j}\left|y_{n-j} - \bar y_j q\right| + \left|y_j y_{n-j} - {n \choose j}^2 q \right| < {n \choose j}^2 -|y_j|^2 ;
        \]
        \item[$(5)$] for all $j = 1, \dots, \left[\frac{n}{2}\right]$,
        \[
        |y_j|^2 - |y_{n-j}|^2 + {n \choose j}^2|q|^2 + 2{n \choose j}\left|y_{n-j} - \bar y_j q\right| < {n \choose j}^2 \text{ and } \; |y_{n-j}|< {n \choose j}
        \]
        \item[$(5')$]for all $j = 1, \dots, \left[\frac{n}{2}\right]$,
        \[
        |y_{n-j}|^2 - |y_j|^2 + {n \choose j}^2|q|^2 + 2{n \choose j}\left|y_j - \bar y_{n-j} q\right| < {n \choose j}^2 \text{ and } \; |y_j|< {n \choose j} ;
        \]
        \item[$(6)$] $|q|<1$ and $|y_j|^2 + |y_{n-j}|^2 - {n \choose j}^2|q|^2 + 2\left| y_jy_{n-j} - {n \choose j}^2 q \right| < {n \choose j}^2$ for all $j = 1, \dots, \left[\frac{n}{2}\right]$;\\
        \item[$(7)$] $\left| y_{n-j} - \bar y_j q \right| + \left| y_j -\bar y_{n-j} q\right| < {n \choose j} (1 - |q|^2)$ for all $j = 1, \dots, \left[\frac{n}{2}\right]$;\\
        \item[$(8)$] there exist $\left[\frac{n}{2}\right]$ number of $2 \times 2$ matrices $B_1,\dots, B_{\left[\frac{n}{2}\right]}$ such that $\n B_j \n < 1$, $y_j = {n \choose j}[B_j]_{11} $, $y_{n-j} = {n \choose j}[B_j]_{22}$ for all $j = 1, \dots, \left[\frac{n}{2}\right]$ and
        \[\det B_1= \det B_2 = \cdots = \det B_{\left[\frac{n}{2}\right]}= q \; ;\]
        \item[$(9)$] there exist $\left[\frac{n}{2}\right]$ number of $2 \times 2$ symmetric matrices $B_1,\dots, B_{\left[\frac{n}{2}\right]}$ such that $\n B_j \n < 1$, $y_j = {n \choose j}[B_j]_{11} $, $y_{n-j} = {n \choose j}[B_j]_{22}$ for all $j = 1, \dots, \left[\frac{n}{2}\right]$ and \[\det B_1= \det B_2 = \cdots = \det B_{\left[\frac{n}{2}\right]}= q \; .\]
    \end{enumerate}
\end{thm}

\vspace{0.3cm}

\begin{proof}
	First we prove the equivalence of the conditions $(2)-(6)$ and $(8), (9)$.\\
	
\noindent \textbf{Case-A:}	We first consider the case when $y_jy_{n-j} = {n \choose j}^2 q$. In this case, each of the conditions $(2)-(6)$ and $(8), (9)$ is equivalent to the pair of conditions $|y_j| < {n \choose j}$ and $|y_{n-j}|< {n \choose j}$.
	\begin{enumerate}
		\item[$(2)$]: Since we have
		\[
		0 \neq {n \choose j} - y_j z - y_{n-j}w + { n \choose j} qzw = {n \choose j} \left(1 -\dfrac{ y_j}{{n \choose j}} z\right)\left(1 -\dfrac{ y_{n-j}}
		{{n \choose j}} w\right)\,,
		\]
		the condition (2) is equivalent to (1). Otherwise if $|y_j | \geq {n \choose j}|$, we may take $z= {n \choose j} y_j^{-1} \in \overline{\mathbb D}$ or if $|y_{n-j}| \geq {n \choose j}$, we may take $w= {n \choose j} y_{n-j}^{-1} \in \overline{\mathbb D}$ to reach a contradiction to the equation (2).
		
		\item[$(3)$]: For the case $y_jy_{n-j} = {n \choose j}^2 q$, we have $\Phi_j(.,y)
		= \dfrac{y_j}{{n \choose j}}$. Thus $\n \Phi_j(.,y) \n < 1$ if and only if
		$|y_j|< {n \choose j}$.
		
		\item[$(3')$]: In this case, $\Phi_{n-j}(.,y) = \dfrac{y_{n-j}}{{n \choose j}}$.
		Thus $\n \Phi_{n-j}(.,y) \n < 1$ if and only if $|y_{n-j}|< {n \choose j}$.
		
		\item[$(4)$]: Since $y_jy_{n-j} = {n \choose j}^2 q$, the statement $(4)$ is equivalent to
		\begin{equation}\label{1}
		{n \choose j}\left|y_j - \bar y_{n-j} q\right| < {n \choose j}^2 -|y_{n-j}|^2
		 \Lra  \; \dfrac{|y_j|}{{n \choose j}}\left| {n \choose j}^2 - |y_{n-j}|^2 \right| < {n \choose j}^2 - |y_{n-j}|^2
		\end{equation}
		which is equivalent to $|y_{n-j}| < {n \choose j} \; \text{ and }\; |y_j|< {n \choose j}$. The reason of the forward implication is that, if $|y_{n-j}| > {n \choose j}$,
		then the left hand side of the inequality $\eqref{1}$ becomes positive while
		the right hand side becomes negative, a contradiction. Also, the inequality
		$\eqref{1}$ is not valid for the case $|y_{n-j}| = {n \choose j}$.
		Hence $|y_{n-j}| < {n \choose j}$ and consequently, by $\eqref{1}$, $|y_j| < {n \choose j}$.
		
		\item[$(4')$]: Similarly, statement $(4')$ is equivalent to the pair
		of conditions $|y_j| < {n \choose j} $ and $|y_{n-j}|< {n \choose j}$.
		
		\item[$(5)$]:
		\begingroup
		\allowdisplaybreaks
		\begin{align}\label{3}
		\nonumber & |y_j|^2 - |y_{n-j}|^2 + {n \choose j}^2|q|^2 +
		2{n \choose j}\left|y_{n-j} - \bar y_j q\right| < {n \choose j}^2\\
		\Lra & \; -\left( 1 - \dfrac{|y_j|^2}{{n \choose j}^2} \right)\left({n \choose j}^2+ |y_{n-j}|^2\right)
		+ 2{n \choose j}|y_{n-j}|\left|1 - \dfrac{|y_j|^2}{{n \choose j}^2} \right| < 0\\
		\nonumber \imp & |y_j|< \nj.
		\end{align}
		\endgroup
		Also by hypothesis, $|y_{n-j}|< {n \choose j}$. Hence statement $(5)$ implies $|y_j| < {n \choose j}$ and $|y_{n-j}|< {n \choose j}$, when $y_j y_{n-j}= \nj^2 q$. On the other hand
		\begingroup
		\allowdisplaybreaks
		\begin{align*}
		& |y_j| < {n \choose j}\\
		\imp & \; -\left( 1 - \dfrac{|y_j|^2}{{n \choose j}^2} \right)\left({n \choose j} - |y_{n-j}|\right)^2 < 0\\
		\imp & \; -\left( 1 - \dfrac{|y_j|^2}{{n \choose j}^2} \right)\left({n \choose j}^2+ |y_{n-j}|^2\right) + 2{n \choose j}|y_{n-j}|\left|1 - \dfrac{|y_j|^2}{{n \choose j}^2} \right| < 0\\
		\Lra &\; |y_j|^2 - |y_{n-j}|^2 + {n \choose j}^2|q|^2 + 2{n \choose j}\left|y_{n-j} - \bar y_j q\right| < {n \choose j}^2.
		\end{align*}
		\endgroup
		Hence $|y_j| < {n \choose j}$ and $|y_{n-j}|< {n \choose j}$ together imply statement $(5)$ for the case $y_j y_{n-j} = \nj^2 q$.
		
		\item[$(5')$]: Similarly, statement $(5')$ is equivalent to the pair of conditions $|y_j| < {n \choose j} $ and $|y_{n-j}|< {n \choose j}$.
		
		\item[$(6)$]: The fact that $|y_j|^2 + |y_{n-j}|^2 - {n \choose j}^2|q|^2 + 2\left| y_jy_{n-j} - {n \choose j}^2 q \right| < {n \choose j}^2$ is equivalent to ${n \choose j}^2 \left(1 - \dfrac{|y_j|^2}{{n \choose j}^2}\right) + |y_{n-j}|^2\left(1 - \dfrac{|y_j|^2}{{n \choose j}^2}\right)>0$, which holds if and only if $${n \choose j}^2\left(1 - \dfrac{|y_j|^2}{{n \choose j}^2}\right) \left(1 - \dfrac{|y_{n-j}|^2}{{n \choose j}^2}\right) >0 .$$ Also $|q|<1 \Lra |y_j||y_{n-j}|<{n \choose j}^2$. Hence in this case, statement $(6)$ holds if and only if
		\begin{equation}\label{5}
		-{n \choose j}^2\left(1 - \dfrac{|y_j|^2}{{n \choose j}^2}\right) \left(1 - \dfrac{|y_{n-j}|^2}{{n \choose j}^2}\right) <0 \q \text{and} \q |y_j||y_{n-j}|<{n \choose j}^2.
		\end{equation}
		The first inequality of $\eqref{5}$ holds only if $\left(1 - \dfrac{|y_j|^2}{{n \choose j}^2}\right)$ and $\left(1 - \dfrac{|y_{n-j}|^2}{{n \choose j}^2}\right)$ have same sign (positive or negative). Then the second inequality of $\eqref{5}$ leads to the fact that both the inequalities in $\eqref{5}$ hold only if $|y_{n-j}| < {n \choose j}$ and $|y_j|< {n \choose j}$. If $|y_{n-j}| < {n \choose j}$ and $|y_j|< {n \choose j}$ hold then clearly $\eqref{5}$ holds. Therefore, the statement $(5)$ holds if and only if $|y_{n-j}| < {n \choose j}$ and $|y_j|< {n \choose j}$.
		
		\item[$(8)$]: Suppose condition $(8)$ holds. Then the matrix $B_j$, as mentioned in the statement of $(8)$, is of the form
		\[
		B_j = \begin{bmatrix}
		\dfrac{y_j}{{n\choose j}} & a_j \\
		\\
		b_j & \dfrac{y_{n-j}}{{n\choose j}}
		\end{bmatrix},
		\]
		where $a_j,b_j \in \C$ such that $\det B_j + a_jb_j = \dfrac{y_jy_{n-j}}{{n \choose j}^2}$. Then $a_j b_j= \dfrac{y_jy_{n-j}}{{n \choose j}^2} -\det B_j =\dfrac{y_jy_{n-j}}{{n \choose j}^2} -q = 0$. Therefore, $B_j$ is one of the following matrices
		\[
		\begin{bmatrix}
		\dfrac{y_j}{{n\choose j}} & * \\
		\\
		0 & \dfrac{y_{n-j}}{{n\choose j}}
		\end{bmatrix} \q \text{or} \q
		\begin{bmatrix}
		\dfrac{y_j}{{n\choose j}} & 0 \\
		\\
		* & \dfrac{y_{n-j}}{{n\choose j}}
		\end{bmatrix}.
		\]
		In any case $\n B_j \n < 1$ if and only if $|y_{n-j}| < {n \choose j}$ and $|y_j|< {n \choose j}$.
		
		\item[$(9)$]:
		Suppose condition $(9)$ holds. Then we must have
		\[
		B_j = \begin{bmatrix}
		\dfrac{y_j}{{n\choose j}} & 0 \\
		\\
		0 & \dfrac{y_{n-j}}{{n\choose j}}
		\end{bmatrix}.
		\]
		Note that, $\n B_j \n < 1$ if and only if $|y_{n-j}| < {n \choose j}$ and $|y_j|< {n \choose j}$.\\
	\end{enumerate}
	
\noindent \textbf{Case-B:}	Next we consider the other case when $y_jy_{n-j} \neq {n \choose j}^2 q$.
	First we prove the following :
	\[
	\begin{array}[c]{ccccc}
	(2)&\Lra&(3)&\Lra&(4)\\
	&  & \Updownarrow & & \\
	&  & (5)          & &
	\end{array}
	\q \text{ and } \q
	\begin{array}[c]{ccccc}
	(2)&\Lra&(3')&\Lra&(4')\\
	&  & \Updownarrow & & \\
	&  & (5')          & &
	\end{array}
	\]
	
	\noindent $(2)\Lra (3)$:
	Fix a $j \in \{1,\dots,n-1\}$. Condition $(2)$ then implies that
	\[
	1 -\dfrac{y_j}{{n \choose j}}z -\dfrac{y_{n-j}}{{n \choose j}} w + qzw \neq 0 \q \text{for all } z,w\in \overline{\mathbb{D}} ,
	\]
	which is equivalent to
	\[
	z\left(\dfrac{y_j}{{n \choose j}} - qw\right) \neq 1- \dfrac{y_{n-j}}{{n \choose j}}w \q \text{for all }  z,w\in \overline{\mathbb{D}} .
	\]
	The above mentioned condition holds only if $1- \dfrac{y_{n-j}}{{n \choose j}}w \neq 0$ for all $w \in \overline\D$, otherwise the choice $z = 0$ will lead to a contradiction. If $|y_{n-j}|< {n \choose j}$, then $\left|\dfrac{y_{n-j}}{{n \choose j}}w \right| < 1$ for all $w \in \overline\D$ and consequently $1- \dfrac{y_{n-j}}{{n \choose j}}w \neq 0$ for all $w \in \overline\D$. If $|y_{n-j}|\geq {n \choose j}$, then for $w= \dfrac{{n \choose j}}{y_{n-j}} \in \overline{\D}$ we have  $1- \dfrac{y_{n-j}}{{n \choose j}}w = 0$. Hence, $1- \dfrac{y_{n-j}}{{n \choose j}}w \neq 0$ for all $w \in \overline\D$ if and only if $|y_{n-j}|< {n \choose j}$. Therefore, condition $(1)$ holds if and only if
	\[
	|y_{n-j}|<{n \choose j}  \text{ and } 1\neq \dfrac{z\left(\dfrac{y_j}{{n \choose j}} - qw\right)}{1- \dfrac{y_{n-j}}{{n \choose j}}w} = z\Phi_j(w,y) \q \text{for all } z,w \in \overline \D,
	\]
	which is equivalent to
	\[
	|y_{n-j}|<{n \choose j}  \text{ and }  1\notin z\Phi_j(\overline{\mathbb{D}},y) \q \text{for all} \; z \in \overline{\mathbb{D}} .
	\]
	If $|y_{n-j}| < \nj$, then the function $\Phi_j(.,y)$ is bounded on $\D$. Then the fact that $1\notin z\Phi_j(\overline{\mathbb{D}},y)$ for all $z \in \overline{\mathbb{D}}$ implies that $\Phi_j(\overline{\mathbb{D}},y)$ does not intersect the complement of $\D$. Otherwise, there exists some $\alpha \in \overline \D$ such that $|\Phi_j(\alpha,y)| \geq 1$ and by choosing $z = \dfrac{1}{\Phi_j(\alpha,y)} \in \overline \D$ we will arrive at a contradiction. Thus the conditions $ |y_{n-j}| < \nj $ and $1\notin z\Phi_j(\overline{\mathbb{D}},y)$ for all $z \in \overline{\mathbb{D}}$ together imply that $\Phi_j(\overline{\mathbb{D}},y) \cap \D^c = \phi$. On the other hand if $\Phi_j(\overline{\mathbb{D}},y) \cap \D^c = \phi$, then $\Phi_j(.,y)$ is bounded on $\overline \D$ hence $|y_{n-j}| < \nj$, moreover $|\Phi_j(\alpha, y)|< 1$ for all $\alpha \in \overline\D$ and hence $|z\Phi_j(\alpha, y)| < 1$ for all $z,\alpha \in \overline\D$. Therefore, condition $(2)$ holds if and only if $\Phi_j(\overline{\mathbb{D}},y)$ does not intersect the complement of $\D$, for each $j = 1, \dots, \left[\frac{n}{2}\right]$. Note that $\Phi_j(\overline{\mathbb{D}},y) \cap \D^c = \phi$ if and only if $\n \Phi_j(.,y) \n < 1$. Thus condition $(2)$ holds if and only if condition $(3)$ holds.\\
	
	\noindent $(2)\Lra (3')$:
	Fixed a $j \in \{1,\dots,n-1\}$. Since condition $(2)$ is equivalent to
	\[
	w\left(\dfrac{y_{n-j}}{{n \choose j}} - qz\right) \neq 1- \dfrac{y_j}{{n \choose j}}z \q \text{for all }  z,w\in \overline{\mathbb{D}} ,
	\]
	by a similar argument as above, condition $(2)$ holds if and only if
	\[
	|y_j|<{n \choose j}  \text{ and }  1\notin w\Phi_{n-j}(\overline{\mathbb{D}},y) \q \text{for all} \; w \in \overline{\mathbb{D}} ,
	\]
	that is if and only if $\Phi_{n-j}(\overline{\mathbb{D}},y)$ does not intersect the complement of $\D$ which is same as saying $\n \Phi_{n-j}(.,y) \n < 1$.\\
	
	\noindent $(3) \Lra (4)$: Note that
	\[
	\n \Phi_j(.,y) \n_{H^{\infty}} = D_j(y) = \dfrac{{n \choose j} \left|y_j - \bar y_{n-j} q \right| + \left|y_j y_{n-j} - {n \choose j}^2 q \right|}{{n \choose j}^2 - |y_{n-j}|^2 }.
	\]
	Therefore $\n \Phi_j(.,y) \n_{H^{\infty}} < 1$ if and only if \[
	{n \choose j} \left|y_j - \bar y_{n-j} q \right| + \left|y_j y_{n-j} - {n \choose j}^2 q \right| < {n \choose j}^2 - |y_{n-j}|^2.
	\]
	Thus conditions $(3)$ and $(4)$ are equivalent.\\
	
	\noindent $(3') \Lra (4')$: Similarly, this equivalence is clear from the following fact
	\[
	\n \Phi_{n-j}(.,y) \n_{H^{\infty}} = D_{n-j}(y) = \dfrac{{n \choose j} \left|y_{n-j} - \bar y_j q \right| + \left|y_j y_{n-j} - {n \choose j}^2 q \right|}{{n \choose j}^2 - |y_j|^2 }.
	\] 
	
	\noindent $(3) \Lra (5)$:
	Fix a $j \in \{1,\dots,n-1\}$. Suppose condition $(3)$ holds. Since $y_jy_{n-j} \neq {n \choose j}^2 q$, the inequality $\n \Phi_j(.,y) \n < 1$ (that is $\Phi_j(.,y)$ is bounded on $\D$) implies $|y_{n-j}| < {n \choose j}$. Again if
	$|y_{n-j}| < {n \choose j}$ holds, then $y_{n-j}z - {n \choose j} \neq 0$ for any $z \in \overline\D$ and consequently the inequality
	\[
	\left|{n \choose j} qz - y_j \right| < \left|y_{n-j}z - {n \choose j}\right| \q \text{ for all } \; z \in \overline{\D}
	\]
	implies $\n \Phi_j(.,y) \n < 1$. Then by the maximum modulus principle, condition $(2)$ holds if and only if
	\[
	|y_{n-j}| < {n \choose j} \q \text{ and } \q \left|{n \choose j} qz - y_j \right| < \left|y_{n-j}z - {n \choose j} \right| \q \text{ for all } \; z \in \T.
	\]
	Note that:
	\begingroup
	\allowdisplaybreaks
	\begin{align}\label{6}
	\nonumber & \left|{n \choose j} qz - y_j \right| < \left|y_{n-j}z - {n \choose j} \right| \q \text{ for all } \; z \in \T \\
	\nonumber \Lra & \; {n \choose j}^2|q|^2 + |y_j|^2 - |y_{n-j}|^2 + 2 {n \choose j}\operatorname{Re}\left(z(y_{n-j}- \bar y_j q )\right) < {n \choose j}^2 \q \text{ for all } \; z \in \T \\
	\Lra & \; |y_j|^2 - |y_{n-j}|^2 +{n \choose j}^2|q|^2  +2 {n \choose j}|y_{n-j}- \bar y_j q | < {n \choose j}^2 \; .
	\end{align}
	\endgroup
	(since $|x| < k $ if and only if $\operatorname{Re}(zx)< k$ for all $z \in \mathbb{T}$.)
	Hence condition $(3)$ holds if and only if $|y_{n-j}| < {n \choose j}$ and inequality $\eqref{6}$ hold, that is, if and only if condition $(5)$ holds.\\
	
	\noindent $(3') \Lra (5')$:
	Using the same line of argument we can show it.
	
	Thus for the case $y_jy_{n-j} \neq {n \choose j}^2 q$ conditions $(2), (3),(3'),(4), (4'),(5),(5')$ are equivalent. To complete the first part of the proof we show the equivalence of $(2),(6),(8)$ and $(9)$.
%	\[
%	\begin{array}[c]{ccccc}
%	(1)&\Lra&(5)&\Leftarrow&(6)\\
%	&  & \Downarrow & \Updiagarrow & \\
%	&  & (7)          & &
%	\end{array}
%	\]
	
	\noindent $(2) \Lra (6)$: write $\hat y = (\hat y_1, ...,\hat y_{n-1}, \widehat q)$, where $\hat y_{_k} = y_{_k}$ for all $k \neq n-j$, $\hat y_{n-j} = {n \choose j}\bar q$ and $\widehat q = \dfrac{\bar y_{n-j}}{{n \choose j}}$. That is,
	\[
	\hat y = \left(y_1,\dots, y_{n-j-1},{n \choose j}\bar q, y_{n-j+1},\dots , y_{n-1}, \dfrac{\bar y_{n-j}}{{n \choose j}} \right).
	\]
	First note that replacing $y_{n-j}$ and $q$ by ${n \choose j}\bar q$ and $\dfrac{\bar y_{n-j}}{{n \choose j}}$ respectively, makes no difference in the inequality in condition $(5')$. Thus, condition $(5')$ holds for the point $y$ ( and for the chosen $j$) if and only if $(5')$ holds for the point $\hat y$. Since we already proved the equivalence of $(2)$, $(5)$ and $(5')$, we can conclude that the inequality in condition $(2)$ holds for the point $y$ if and only if the inequalities in $(5)$ hold for the point $\hat y$. Thus, condition $(2)$ holds if and only if the following holds :
	\[
	|y_j|^2 - \left| {n \choose j} \bar q \right|^2 + \nj^2 \left| \dfrac{\bar y_{n-j}}{{n \choose j}} \right|^2 + 2 \nj \left| {n \choose j}\bar q - \bar y_j \dfrac{\bar y_{n-j}}{{n \choose j}} \right| < \nj^2  \text{and} \q \left|{n \choose j} \bar q \right|< {n \choose j},
	\]
	which is same as saying that
	\[
	|y_j|^2 + |y_{n-j}|^2 - {n \choose j}^2|q|^2  + 2\left|y_j y_{n-j} - {n \choose j}^2  q \right| < {n \choose j}^2 \q \text{and} \q |q|<1.
	\]
	Hence $(2) \Lra (6)$. To complete the rest of the proof we need the following lemma, which can be proved by simple calculations.
	% % % % % % % % % % % % % % % % % % % % % % % % % % % % % % % % % % % % % % % % % % % % % % % % % % % % % % % % % % % % % % % % %
	% %----------- Lemma for Determinant of matrix------- % %
	\begin{lemma}\label{matrix-det}
		If
		\begin{equation*}
		B_j = \begin{bmatrix}
		\dfrac{y_j}{{n \choose j}} & b_j \\
		\\
		c_j & \dfrac{y_{n-j}}{{n \choose j}}
		\end{bmatrix}
		\end{equation*}
		where $b_jc_j = \dfrac{y_jy_{n-j}}{{n \choose j}^2} -q $. Then $ \det B_j = q$,
		\begin{equation}\label{matrix-B}
		1-B_j^*B_j =
		\begin{bmatrix}
		1-\dfrac{|y_j|^2}{{n \choose j}^2} - |c_j|^2 & & -b_j\dfrac{\bar y_j}{{n \choose j}} - \bar c_j\dfrac{y_{n-j}}{{n \choose j}} \\
		\\
		-\bar b_j\dfrac{y_j}{{n \choose j}} - c_j\dfrac{\bar y_{n-j}}{{n \choose j}} & & 1-\dfrac{|y_{n-j}|^2}{{n \choose j}^2} - |b_j|^2
		\end{bmatrix}
		\end{equation}
		and
		\begin{equation}\label{det-B}
		\det (1-B_j^*B_j) = 1-\dfrac{|y_j|^2}{{n \choose j}^2} -\dfrac{|y_{n-j}|^2}{{n \choose j}^2} +|q|^2 - |b_j|^2 - |c_j|^2 \, .
		\end{equation}
	\end{lemma}
	
	\noindent Let us get back to the proof of Theorem \ref{char G 3}.\\
	
	\noindent$(6)\imp (9)$:
	Suppose $(6)$ holds. Consider the matrix
	\begin{equation*}
	B_j = \begin{bmatrix}
	\dfrac{y_j}{{n \choose j}} & k_j \\
	\\
	k_j & \dfrac{y_{n-j}}{{n \choose j}}
	\end{bmatrix}
	\end{equation*}
	where $k_j$ is the square root(any) of $\left(\dfrac{y_jy_{n-j}}{{n \choose j}^2} -q\right) $. Then $\det B_j = q$. By Lemma $\ref{matrix-det}$, the diagonals of $1 - B_j^*B_j$ are the following :
	\[
	1-\dfrac{|y_j|^2}{{n \choose j}^2} - \left|\dfrac{y_jy_{n-j}}{{n \choose j}^2} -q\right| \; \text{ and } \; 1-\dfrac{|y_{n-j}|^2}{{n \choose j}^2} - \left|\dfrac{y_jy_{n-j}}{{n \choose j}^2} -q\right|.
	\]
	We have proved that condition $(2)$ is equivalent to conditions $ (4), (4')$ and $(6)$. Thus if condition $(6)$ holds, then conditions $(4)$ and $(4')$ also hold. Then by conditions $(4)$ and $(4')$, we have
	\[
	1-\dfrac{|y_j|^2}{{n \choose j}^2} - \left|\dfrac{y_jy_{n-j}}{{n \choose j}^2} -q\right|> 0 \; \text{ and } \; 1-\dfrac{|y_{n-j}|^2}{{n \choose j}^2} - \left|\dfrac{y_jy_{n-j}}{{n \choose j}^2} -q\right|> 0,
	\]
	that is, the diagonals of $1 - B_j^*B_j$ are positive. Again by Lemma $\ref{matrix-det}$ and condition $(6)$, we have
	\[
	\det (1-B_j^*B_j)=  1-\dfrac{|y_j|^2}{{n \choose j}^2} -\dfrac{|y_{n-j}|^2}{{n \choose j}^2} +|q|^2 - 2 \left|\dfrac{y_jy_{n-j}}{{n \choose j}^2} -q\right| > 0.
	\]
	Thus $\n B_j \n < 1$. Also we have $y_j = {n \choose j}[B_j]_{11} $ and $y_{n-j} = {n \choose j}[B_j]_{22}$. Hence condition $(9)$ is satisfied.\\
	
	\noindent $(9) \imp (8)$ : Obvious.\\
	
	\noindent $(8) \imp (6)$ : Suppose $(8)$ holds. Then  $B_j$ is a $2 \times 2$ matrix such that $[B_j]_{11}= \dfrac{y_j}{{n \choose j}}$, $[B_j]_{22}= \dfrac{y_{n-j}}{{n \choose j}}$ and $[B_j]_{12}[B_j]_{21} = \dfrac{y_jy_{n-j}}{{n \choose j}^2} -q $. Since $ \left( \left|[B_j]_{12}\right| - \left| [B_j]_{21} \right| \right)^2 \geq 0 $, we have
	\[
	\left|[B_j]_{12}\right|^2 + \left|[B_j]_{21}\right|^2 \geq 2 \left|[B_j]_{12}[B_j]_{21}\right| = 2 \left|\dfrac{y_jy_{n-j}}{{n \choose j}^2} -q\right|.
	\]
	Since $\n B_j \n < 1$, by lemma $\ref{matrix-det}$ we have
	\begin{align*}
	& 1-\dfrac{|y_j|^2}{{n \choose j}^2} -\dfrac{|y_{n-j}|^2}{{n \choose j}^2} +|q|^2 - 2 \left|\dfrac{y_jy_{n-j}}{{n \choose j}^2} -q\right|\\
	\geq & \; 1-\dfrac{|y_j|^2}{{n \choose j}^2} -\dfrac{|y_{n-j}|^2}{{n \choose j}^2} +|q|^2 - |[B_j]_{12}|^2 - |[B_j]_{21}|^2 \\
	= & \; \det (1-B_j^*B_j) \\
	> & \; 0.
	\end{align*}
	Therefore,
	\[
	|y_{n-j}|^2 + |y_j|^2 - {n \choose j}^2|q|^2 + 2\left|y_j y_{n-j}- \nj^2 q\right| < {n \choose j}^2 .
	\]
	The fact $\n B_j \n < 1$ implies $|\det B_j| <1 $, and thus $|q| < 1$. Hence $(6)$ holds and the first part of the proof is complete.\\
	
To complete the rest of the proof we show the following : $(1)\Rightarrow(2) \Rightarrow (7)\Rightarrow (1)$.\\
%\[
%\begin{array}[c]{ccc}
%(2)&\Leftarrow&(1)\\
%\Downarrow & \Updiagarrow & \\
%(7)
%\end{array}
%\]
\noindent $(1) \imp (2)$:
Suppose $y =(y_1,\dots, y_{n-1},q) \in \Gn$. Then $|q| < 1$, and there exists $(\be_1,\dots,\be_{n-1}) \in \C^{n-1}$ such that for each $j \in \left\{ 1, \dots, \left[\frac{n}{2}\right] \right\}$
\[
|\be_j|+ |\be_{n-j}| < \nj, \q y_j = \be_j + \bar \be_{n-j} q \q \text{and}\q y_{n-j} = \be_{n-j} + \bar \be_j q.
\]
Hence
\[
|y_{n-j}| = |\beta_{n-j} + \bar\beta_j q| \leq |\beta_{n-j}| + |\beta_j| < {n \choose j}.
\]
Then, for each $j=1,\dots, [\frac{n}{2}]$,
\begingroup
\allowdisplaybreaks
\begin{align*}
|y_j|^2 -|y_{n-j}|^2 & = (\beta_j + \bar\beta_{n-j} q)(\bar\beta_j + \beta_{n-j} \bar q) - (\beta_{n-j} + \bar\beta_j q)(\bar\beta_{n-j} + \beta_j \bar q)\\
& = |\beta_j|^2(1 - |q|^2) - |\beta_{n-j}|^2(1 - |q|^2)\\
&  < {n \choose j} (|\beta_j| -|\beta_{n-j}|)(1 - |q|^2)
\end{align*}
\endgroup
and
\[
y_{n-j} - \bar y_j q = \beta_{n-j} + \bar\beta_j q - (\bar\beta_j + \beta_{n-j} \bar q)q = \beta_{n-j}(1 - |q|^2).
\]
Hence, for all $j=1,\dots, [\frac{n}{2}]$,
\[
|y_j|^2 -|y_{n-j}|^2 - 2 {n \choose j} |y_{n-j} - \bar y_j q|     < {n \choose j} (|\beta_j| + |\beta_{n-j}|)(1 - |q|^2) < {n \choose j}^ 2 (1 - |q|^2).
\]
Thus, for all $j=1,\dots, [\frac{n}{2}]$,
\begin{equation}\label{charG2 1}
|y_j|^2 -|y_{n-j}|^2 + \nj^2|q|^2 - 2 {n \choose j} |y_{n-j} - \bar y_j q| < \nj^2.
\end{equation}
Therefore, condition $(5)$ holds. Then by equivalence of conditions $(2)$ and $(5)$, we have
\[
{n \choose j} - y_j z - y_{n-j}w + { n \choose j} qzw \neq 0 \q \text{for all } z,w \in \overline\D \;  \text{ and for all } j=1,\dots, [\frac{n}{2}]
\]
Hence $(2)$ holds.\\

\noindent$(2) \imp (7)$:
Suppose condition $(2) $ holds. Then by conditions $(5)$ and $(5')$,
\begin{align*}
& |y_j|^2 - |y_{n-j}|^2 + {n \choose j}^2|q|^2 + 2{n \choose j}\left|y_{n-j} - \bar y_j q\right| < {n \choose j}^2  \\
\text{and} \q  & |y_{n-j}|^2 - |y_j|^2 + {n \choose j}^2|q|^2 + 2{n \choose j}\left|y_j - \bar y_{n-j} q\right| < {n \choose j}^2 .
\end{align*}
By adding the above two inequalities we get
\begin{equation}\label{charG2 2}
\left| y_{n-j} - \bar y_j q \right| + \left| y_j -\bar y_{n-j} q\right| < {n \choose j} (1 - |q|^2).
\end{equation}
Since $j$ is chosen arbitrarily, inequality $\ref{charG2 2}$ holds for each $j \in \left\{ 1, \dots, \left[\frac{n}{2}\right] \right\}$. Hence $(2) \imp (7)$. Further note that if inequation in condition $(2)$ holds for some $j \in \left\{ 1, \dots, \left[\frac{n}{2}\right] \right\}$, then the inequality in condition $(7)$ holds for the same $j$.\\

\noindent $(7) \imp (1)$:
Suppose condition $(7)$ holds. Clearly $|q| <1$. For each $j = 1, \dots, \left[\frac{n}{2}\right] $, consider
\[
\beta_j = \dfrac{y_j - \bar y_{n-j}q}{1 -|q|^2} \q \text{and} \q \beta_{n-j} = \dfrac{y_{n-j} - \bar y_jq}{1 -|q|^2}.
\]
Then
$
\beta_j + \bar{\beta}_{n-j} q = y_j \q \text{and} \q \be_{n-j} + \bar{\be}_j q = y_{n-j} ,
$
for all $j = 1, \dots, \left[\frac{n}{2}\right]$ which is same as saying that $\beta_j + \bar{\beta}_{n-j} q = y_j $,
for all $j = 1, \dots , n-1$. By condition $(7)$ of this theorem, we have
\[ |\beta_j| + |\beta_{n-j}| < {n \choose j}\q \text{for each } j = 1, \dots, \left[\frac{n}{2}\right] .\]
Hence for the point $y =(y_1,\dots, y_{n-1},q) \in \Cn $, there exists $(\be_1, \dots, \be_{n-1}) \in \C^{n-1}$ such that $ y_j = \beta_j + \bar{\beta}_{n-j} q$ and $|\beta_j| + |\beta_{n-j}| < {n \choose j}$ for all $j =  1, \dots, n-1 $. Therefore $y \in \Gn$, and the proof is complete.

\end{proof}

\subsection{Characterizations of a point in $\Gamn$}\label{char Gamn}
In this subsection, we state and prove an analog of Theorem $\ref{char G 3}$ for $\Gamn \,(=\overline{\Gn})$. 

\begin{thm}\label{char F 3}
    For a point $y =(y_1,\dots, y_{n-1},q) \in \Cn$, the following are equivalent:
    \begin{enumerate}
        \item[$(1)$] $y \in \Gamn$;
        \item[$(2)$] ${n \choose j} - y_j z - y_{n-j}w + { n \choose j} qzw \neq 0$, for all $z,w \in \D$ and for all $j = 1, \dots, \left[\frac{n}{2}\right]$;
        \item[$(3)$] for all $j = 1, \dots, \left[\frac{n}{2}\right]$,
        \[
        \n \Phi_j(.,y)\n_{H^{\infty}} \leq 1 \; \text{and if } y_jy_{n-j}= {n \choose j}^2 q \; \text{ then,}
        \;\text{ in addition, }\; |y_{n-j}|\leq {n \choose j}
        \]
        \item[$(3')$] for all $j = 1, \dots, \left[\frac{n}{2}\right]$,
        \[
        \n \Phi_{n-j}(.,y)\n_{H^{\infty}} \leq 1 \; \text{and if } y_jy_{n-j}= {n \choose j}^2 q \; \text{ then,}
        \;\text{ in addition, }\; |y_j|\leq {n \choose j} ;
        \]
        \item[$(4)$] for all $j = 1, \dots, \left[\frac{n}{2}\right]$,
        \begin{align*}
            &{n \choose j}\left|y_j - \bar y_{n-j} q\right| + \left|y_j y_{n-j} - {n \choose j}^2 q \right| \leq {n \choose j}^2 -|y_{n-j}|^2 \; \; \text{ and if } y_jy_{n-j}= {n \choose j}^2 q \\ & \text{ then,}
            \;\text{ in addition, }\; |y_j|\leq {n \choose j}
        \end{align*}
        \item[$(4')$] for all $j = 1, \dots, \left[\frac{n}{2}\right]$
        \begin{align*}
            &{n \choose j}\left|y_{n-j} - \bar y_j q\right| + \left|y_j y_{n-j} - {n \choose j}^2 q \right| \leq {n \choose j}^2 -|y_j|^2 \; \; \text{ and if } y_jy_{n-j}= {n \choose j}^2 q \\ & \text{ then,}
            \;\text{ in addition, }\; |y_{n-j}|\leq {n \choose j} ;
        \end{align*}
        \item[$(5)$] for all $j = 1, \dots, \left[\frac{n}{2}\right]$, 
        \[
        |y_j|^2 - |y_{n-j}|^2 + {n \choose j}^2|q|^2 + 2{n \choose j}\left|y_{n-j} - \bar y_j q\right| \leq {n \choose j}^2 \text{ and } \; |y_{n-j}|\leq {n \choose j}
        \]
        \item[$(5')$] for all $j = 1, \dots, \left[\frac{n}{2}\right]$
        \[
        |y_{n-j}|^2 - |y_j|^2 + {n \choose j}^2|q|^2 + 2{n \choose j}\left|y_j - \bar y_{n-j} q\right| \leq {n \choose j}^2 \text{ and } \; |y_j|\leq {n \choose j} ;
        \]
        \item[$(6)$] $|q|\leq 1$ and $|y_j|^2 + |y_{n-j}|^2 - {n \choose j}^2|q|^2 + 2\left| y_jy_{n-j} - {n \choose j}^2 q \right| \leq {n \choose j}^2$ for all $j = 1, \dots, \left[\frac{n}{2}\right]$;
        \item[$(7)$] $\left| y_{n-j} - \bar y_j q \right| + \left| y_j -\bar y_{n-j} q\right| \leq {n \choose j} (1 - |q|^2)$ for all $j = 1, \dots, \left[\dfrac{n}{2}\right]$, and if $ |q|=1$ then, in addition, $|y_j|\leq {n \choose j}$ for all $j = 1, \dots, \left[\frac{n}{2}\right]$;
        \item[$(8)$] there exist $\left[\frac{n}{2}\right]$ number of $2 \times 2$ matrices $B_1,\dots, B_{\left[\frac{n}{2}\right]}$ such that $\n B_j \n \leq 1$, $y_j = {n \choose j}[B_j]_{11} $, $y_{n-j} = {n \choose j}[B_j]_{22}$ for all $j = 1, \dots, \left[\frac{n}{2}\right]$ and \[\det B_1= \det B_2 =\cdots= \det B_{\left[\frac{n}{2}\right]}= q \; ;\]
        \item[$(9)$] there exist $\left[\frac{n}{2}\right]$ number of $2 \times 2$ symmetric matrices $B_1,\dots, B_{\left[\frac{n}{2}\right]}$ such that $\n B_j \n \leq 1$, $y_j = {n \choose j}[B_j]_{11} $, $y_{n-j} = {n \choose j}[B_j]_{22}$ for all $j = 1, \dots, \left[\frac{n}{2}\right]$ and \[\det B_1= \det B_2 =\cdots= \det B_{\left[\frac{n}{2}\right]}= q. \]

    \end{enumerate}
\end{thm}
\begin{proof}
	\noindent $(1)\Lra (2)$:
	Suppose $(2)$ holds. For any $\zeta , \eta \in \overline{\mathbb{D}} $ and for any $r\in (0,1)$ we have $r\zeta, r\eta \in \D$. Hence
	\[
	{n \choose j} - y_j r\zeta - y_{n-j}r\eta + { n \choose j} qr^2\zeta\eta \neq 0\q \text{for all } j = 1, \dots, \left[\frac{n}{2}\right].
	\]
	Since the above is true for any $\zeta , \eta \in \overline{\mathbb{D}} $ and any $r\in (0,1)$, by Theorem $\ref{char G 3}$, we have \[ (ry_1,ry_2, \dots, ry_{n-1},r^2q) \in \Gn\] for any $r\in (0,1)$. Therefore, $y =(y_1,\dots, y_{n-1},q) \in \overline{\Gn}=\Gamn$.\\
	
	Conversely, suppose $y =(y_1,\dots, y_{n-1},q) \in \Gamn$. Also suppose $(2)$ does not hold, that is, for some $j = 1, \dots, \left[\frac{n}{2}\right]$
	\begin{align*}
	&{n \choose j} - y_j z - y_{n-j}w + { n \choose j} qzw = 0 \q \text{for some } z,w \in \D \\
	\imp & \; z\left({n \choose j} q w - y_j \right) = \left( y_{n-j}w - {n \choose j} \right) \q \text{for some } z,w \in \D .
	\end{align*}
	Since $y \in \Gamn$, we have $|y_{n-j}| \leq \nj$ and hence $y_{n-j}w - {n \choose j} \neq 0$ for all $w \in \D$. Therefore,
	\[ z\dfrac{{n \choose j} q w - y_j }{ y_{n-j}w - {n \choose j} } = z \Phi_j(w,y) = 1 \q \text{for some } z,w \in \D \]
	and consequently $|\Phi_j(w,y)|>1$ for some $w \in \D $. Again by Theorem \ref{char G 3} we have that, $|\Phi_j(w,\zeta)| < 1$ whenever $w \in \D$ and $\zeta \in \Gn $. Since $ y\in \Gamn $, we must have $|\Phi_j(w,y)| \leq 1$ for any $w \in \D $, a contradiction. Hence $(2)$ must holds, if $y \in \Gamn$. \\

	Now we prove the equivalence of $(2)-(6)$ and $(8), (9)$.
	The proof of this part is similar to the proof of the corresponding part of Theorem $\ref{char G 3}$. For the case $y_y y_{n-j} = \nj^2 q$, each condition is equivalent to the pair of statements $|y_j| \leq {n \choose j}$ and $|y_{n-j}|\leq {n \choose j}$. The explanations for each case are similar to the proof (for the corresponding case) of Theorem $\ref{char G 3}$. In this case, we explain the reason for the parts $(4)$ and $(4')$ as they are slightly different form the similar parts in Theorem $\ref{char G 3}$.
	
	\item[$(4)$]: Since $y_jy_{n-j} = {n \choose j}^2 q$, we have the following:
	\begingroup
	\allowdisplaybreaks
	\begin{align}\label{11}
	\nonumber & {n \choose j}\left|y_j - \bar y_{n-j} q\right| \leq {n \choose j}^2 -|y_{n-j}|^2\\
	\nonumber \Lra & \; {n \choose j}\left|y_j - \bar y_{n-j} \dfrac{y_jy_{n-j}}{{n \choose j}^2} \right| \leq {n \choose j}^2 -|y_{n-j}|^2\\
	 \imp & \; |y_{n-j}| \leq {n \choose j}.
	\end{align}
	\endgroup
	Since $|y_{n-j}| > {n \choose j}$ provides a contradiction, as the left hand side of the inequality $\eqref{11}$ becomes positive while the right hand side becomes negative. Hence statement $(3)$ implies $|y_j| \leq \nj$ and $|y_{n-j}| \leq {n \choose j}$. On the other hand
	\begin{align*}
	&|y_j| < {n \choose j} \; \text{ and }\; |y_{n-j}|< {n \choose j}\\
	\imp & \; \dfrac{|y_j|}{{n \choose j}}\left| {n \choose j}^2 - |y_{n-j}|^2 \right| \leq {n \choose j}^2 - |y_{n-j}|^2\\
	\Lra & {n \choose j}\left|y_j - \bar y_{n-j} q\right| \leq {n \choose j}^2 -|y_{n-j}|^2.
	\end{align*}
	That is, $|y_j| \leq \nj$ and $|y_{n-j}| \leq {n \choose j}$ together imply statement $(4)$ for the case  $y_y y_{n-j} = \nj^2 q$.\\
	
	\item[$(4')$]: Similarly, statement $(4')$ is equivalent to the inequality
	\[
	\dfrac{|y_{n-j}|}{{n \choose j}}\left| {n \choose j}^2 - |y_j|^2 \right| < {n \choose j}^2 - |y_j|^2
	\]
	which implies $|y_j| < {n \choose j}$. Hence statement $(4')$ implies $|y_j| \leq \nj$ and $|y_{n-j}| \leq {n \choose j}$. Also, $|y_j| \leq \nj$ and $|y_{n-j}| \leq {n \choose j}$ together imply statement $(4')$ for this case.\\
	
	For the case  $y_y y_{n-j} \neq \nj^2 q$, the proof is similar to the corresponding parts of Theorem $\ref{char G 3}$.
% % % % % ------ fisrt part proof complete ------- % % % % % % %	
	\noindent $(2)\imp (7)$: we can have similar proof as in the corresponding part of Theorem \ref{char G 3}.\\
	
	\noindent $(7)\imp (10)$:
	Suppose condition $(7)$ holds, then $|q| \leq 1$. For $|q|<1$, the part $(7)\Rightarrow(1)$ of the proof of Theorem \ref{char G 3} works here, that is, take
	\[
	\beta_j = \dfrac{y_j - \bar y_{n-j}q}{1 -|q|^2} \q \text{and} \q \beta_{n-j} = \dfrac{y_{n-j} - \bar y_jq}{1 -|q|^2} \q \text{for } j = 1,\dots, \lt \frac{n}{2} \rt
	\]
	to obtain statement $(10)$. If $|q|=1$, then condition $(3)$ implies $y_j = \bar{y}_{n-j} q$ and $y_{n-j} = \bar y_j q$, for all $j = 1, \dots, \left[\frac{n}{2}\right]$. Hence  $|y_{n-j}|=|y_j| \leq \nj$. Choose $r_1,\dots, r_{[\frac{n}{2}]} \in [0,1]$. Consider
	\[\be_j = r_j y_j \q \text{and} \q \be_{n-j} = (1-r_j) y_{n-j} \q \text{for all } j = 1, \dots, \left[\frac{n}{2}\right]. \]
	Then for all $j = 1, \dots, \left[\frac{n}{2}\right]$ we have the following
	\begin{align*}
	& |\be_j| + |\be_{n-j}| \leq {n \choose j}r_j + {n \choose j}(1-r_j) = {n \choose j},\\
	& \be_j + \bar{\be}_{n-j} q = r_j y_j + (1-r_j) \bar{y}_{n-j}q = r_j y_j + (1-r_j) y_j = y_j \\
	\text{and} \q & \be_{n-j} + \bar{\be}_j q = (1-r_j) y_{n-j} + r_j \bar{y}_jq = (1-r_j) y_{n-j} + r_j y_{n-j} = y_{n-j}.
	\end{align*}
	Thus for either cases, there exists $(\be_1, \dots, \be_{n-1}) \in \C^n$ such that $ y_j = \beta_j + \bar{\beta}_{n-j} q$, $y_{n-j}= \be_{n-j} + \bar\be_j q$ and $|\beta_j| + |\beta_{n-j}| \leq {n \choose j}$ for all $j =  1, \dots, \left[\frac{n}{2}\right] $. Hence $(10)$ holds.\\
	
	\noindent $(10)\imp (2)$:
	Suppose $(10)$ holds. Then $|q| \leq 1$, and there is $(\be_1,\dots,\be_{n-1}) \in \C^n$ such that for any $j \in \left\{ 1, \dots, \left[\frac{n}{2}\right] \right\}$
	\[
	|\be_j|+ |\be_{n-j}| \leq \nj, \q y_j = \be_j + \bar \be_{n-j} q \q \text{and}\q y_{n-j} = \be_{n-j} + \bar \be_j q.
	\]
	Hence
	\[
	|y_{n-j}| = |\beta_{n-j} + \bar\beta_j q| \leq |\beta_{n-j}| + |\beta_j| \leq {n \choose j}.
	\]
	Let $j \in \left\{ 1, \dots, \left[\frac{n}{2}\right] \right\}$ be arbitrary. Then
	\begingroup
	\allowdisplaybreaks
	\begin{align*}
	|y_j|^2 -|y_{n-j}|^2
	= & \; |\beta_j|^2 +|\beta_{n-j}|^2|q|^2 + 2 \operatorname{Re}(\beta_j \beta_{n-j} \bar q) -|\beta_{n-j}|^2 - |\beta_j|^2 |q|^2 - 2 \operatorname{Re}(\beta_j \beta_{n-j} \bar q)\\
	=& \; (|\beta_j|^2 -|\beta_{n-j}|^2)(1 - |q|^2) \\
	\leq & \; {n \choose j} (|\beta_j| -|\beta_{n-j}|)(1 - |q|^2),
	\end{align*}
	\endgroup
	and $ \quad
	y_{n-j} - \bar y_j q = \beta_{n-j} + \bar\beta_j q - (\bar\beta_j + \beta_{n-j} \bar q)q = \beta_{n-j}(1 - |q|^2).
	$
	Hence
	\[
	|y_j|^2 -|y_{n-j}|^2 - 2 {n \choose j} |y_{n-j} - \bar y_j q|
	\leq {n \choose j} (|\beta_j| + |\beta_{n-j}|)(1 - |q|^2)\leq {n \choose j}^ 2 (1 - |q|^2),
	\]
	that is,
	\[
	|y_j|^2 -|y_{n-j}|^2 + \nj^2|q|^2 - 2 {n \choose j} |y_{n-j} - \bar y_j q| \leq \nj^2.
	\]
	Thus, condition $(5)$ holds. 
%	Then by the equivalence of conditions $(1)$ and $(4)$ of Theorem $\ref{char F 1}$, we have
%	\[
%	{n \choose j} - y_j z - y_{n-j}w + { n \choose j} qzw \neq 0 \q \text{for all } z,w \in \D.
%	\]
	Since $j$ was chosen arbitrarily and condition $(2)$ is equivalent to $(5)$,
%	the above inequation holds for all $j = 1, \dots, \left[\frac{n}{2}\right] $. 
	thus condition $(2)$ of this theorem holds. Consequently, the proof is complete.
\end{proof}

\begin{rem}\label{rem:21}
The set $\widetilde{\Gamma}_n$ can be described as follows :

\begin{align*}
\widetilde{\Gamma}_n := \Bigg\{ (y_1,\dots,y_{n-1}, q)\in \C^n :\; q \in \overline\D, \;  y_j = \be_j + \bar \be_{n-j} q, \; \beta_j \in \C &\text{ and }\\
|\beta_j|+ |\beta_{n-j}| \leq {n \choose j} &\text{ for } j=1,\dots, n-1 \Bigg\}.
\end{align*}

\end{rem}

\begin{rem} \label{rem:22}
	Let $y=(y_1,\dots,y_{n-1},q) \in \Cn$. For any $j \in \left\{1,\dots, \left[\frac{n}{2}\right] \right\}$, if $\tilde y^j = (\tilde y_1,\dots,\tilde y_{n-1},\tilde q)$ where
	$\tilde y_k =y_k$ for $k\neq j,n-j;\q \tilde y_j = y_{n-j}, \; \; \tilde y_{n-j} = y_j$ and $\tilde q = q$,
	then $y \in \Gn($or $\Gamn)$ if and only if $\tilde y^j \in \Gn($or $\Gamn)$.	
	
\end{rem}

\begin{rem} \label{rem:23}
It was proved in \cite{ay-jot} (see Theorem 1.1 in \cite{ay-jot}) that the closed symmetrized bidisc $\Gamma_2$ has the following representation:
\[
\Gamma_2 =\{(s,p)\in\mathbb C^2\,:\, s=\beta + \bar{\beta}p, \text{ for any } p, \beta \in \overline{\D}  \}
\]
Thus it is evident from Theorem \ref{char F 3} (and also from Remark \ref{rem:21}) that $\widetilde{\Gamma_2}=\Gamma_2$. We shall see in Section 4 that $\gamn \subsetneq \Gamn$ for all $n\geq 3$.

\end{rem}

\vspace{0.3cm}

\section{Characterizations of the points in $\gn$ and $\gamn$}

\vspace{0.3cm}

\noindent In this section, we find a variety of characterizations for the symmetrized polydisc $\gn$ via the characterizations of $\Gn$ as in Theorem \ref{char G 3}. In \cite{costara1}, Costara found the following interesting characterization of a point in $\gn$ and $\gamn$.

\begin{thm}[\cite{costara1}, Theorems 3.6 and 3.7]\label{gn}
	For a point $(s_1,\dots,s_{n-1},p) \in \C^n$, the following are equivalent:
	\begin{enumerate}
		\item The point $(s_1,\dots,s_{n-1},p) \in \gn \;  ($respectively $\in \Gamma_n)$. 
		\item $|p|<1 \;($respectively $\leq 1)$ and there exists $\lf S_1,\dots, S_{n-1} \rf \in \mathbb G_{n-1} \;($respectively $\in \Gamma_{n-1})$ such that 
		\[ 
		s_j = S_j + \bar S_{n-j} p \q \text{for }\; j=1,\dots, n-1 .
		\]
	\end{enumerate}
\end{thm}
We have mentioned before (see Remark \ref{rem:23}) that $\widetilde{\Gamma}_2=\Gamma_2$. Indeed, it is true that $\Gn=\gn$ for $n=2$ but $\gn \subsetneqq \Gn$ for $n\geq 3$. The first part of this claim follows straight from the above result of Costara (Theorem \ref{gn}). We show below that $\gn \subsetneqq \Gn$ for $n\geq 3$. By Theorem \ref{gn}, $s_i=S_i+\bar{S}_{n-i}p$ for every point $(s_1,\dots ,s_{n-1},p)\in \gn$, where $(S_1,\dots, S_{n-1})\in \mathbb G_{n-1}$. Needless to mention that $|S_i|+|S_{n-i}|< {n \choose i}$ and this implies that $(s_1,\dots, s_{n-1},p)\in \Gn$. Now consider the point $\lf \dfrac{5}{2}, \dfrac{5}{4}, \dfrac{1}{2} \rf \in \C^3$. We can write $\lf \dfrac{5}{2}, \dfrac{5}{4}, \dfrac{1}{2} \rf = \lf \be_1 + \overline\be_2 p, \be_2 + \overline\be_1 p, p \rf$, where $\be_1 = \dfrac{5}{2}, \be_2 = 0, p=\dfrac{1}{2} $. Also $\beta_1 , \beta_2$ are unique because if we write $(y_1,y_2,p)=\lf \dfrac{5}{2}, \dfrac{5}{4}, \dfrac{1}{2} \rf$, then
\[
\beta_1=\frac{y_1-\bar{y}_2p}{1-|p|^2} \; \text{ and } \; \beta_2=\frac{y_2-\bar{y}_1p}{1-|p|^2}.
\]
Since $|\beta_1|+|\beta_2|<3$, $\lf \dfrac{5}{2}, \dfrac{5}{4}, \dfrac{1}{2} \rf \in \widetilde{\mathbb G_3}$. Clearly $\lf \be_1, \be_2 \rf \notin \mathbb G_2$, as $|\be_1| > 2$. By Theorem \ref{gn}, $\lf \dfrac{5}{2}, \dfrac{5}{4}, \dfrac{1}{2} \rf \notin \mathbb G_3$. So, $\mathbb G_3 \subsetneqq \widetilde{\mathbb G_3}$. Similarly for any $n>3$, we may choose $\lf \beta_1,\dots ,\beta_{n-1} \rf=\lf \dfrac{2n-1}{2},0,\dots, 0 \rf$ and $p=\dfrac{1}{2}$ to obtain a point $(y_1,\dots , y_{n-1},p)$, where $y_i=\beta_i+ \bar{\beta}_{n-i}p,$ for $i=1,\dots, n-1$, which is in $\widetilde{\mathbb G_n}$ but not in $\mathbb G_n$.

\subsection{Characterizations of a point in $\gn$}
\begin{thm}\label{gn - Gn}
    Let $(s_1,\dots,s_{n-1},p) \in \C^n$ with $|p|\neq 1$. Consider the point
    $$Q=\lf \dfrac{s_1 - \bar s_{n-1} p}{1-|p|^2}, \dfrac{s_2 - \bar s_{n-2} p}{1 - |p|^2}, \dots, \dfrac{s_{n-1} - \bar s_1 p}{1 - |p|^2} \rf \in \C^{n-1}.$$ Then $(s_1,\dots,s_{n-1},p) \in \gn$ if and only if $(s_1,\dots,s_{n-1},p) \in \Gn$ and $Q \in \mathbb G_{n-1}$.
\end{thm}
\begin{proof}
    Suppose $(s_1,\dots,s_{n-1},p) \in \gn$. Then $|p|<1$ and $(s_1,\dots , s_{n-1},p)\in \Gn$ as $\gn \subseteq \Gn$. By Theorem \ref{gn}, there exists $(\be_1,\dots,\be_{n-1}) \in \mathbb G_{n-1}$ so that
    \[
    s_j = \be_j + \bar \be_{n-j} p \q \text{for each } j = 1, \dots, n-1.
    \]
    Then for each $j = 1, \dots, n-1$, we have
    $
    \be_j = \dfrac{s_j - \bar s_{n-j} p }{1-|p|^2}.
    $
    Also these $\beta_j's$ are unique. So,
 by Theorem \ref{gn},
    \[
    Q=\lf \dfrac{s_1 - \bar s_{n-1} p}{1-|p|^2}, \dfrac{s_2 - \bar s_{n-2} p}{1 - |p|^2}, \dots, \dfrac{s_{n-1} - \bar s_1 p}{1 - |p|^2} \rf = (\be_1,\dots,\be_{n-1}) \in \mathbb G_{n-1}.
    \]
   
    Conversely, suppose $(s_1,\dots,s_{n-1},p) \in \Gn$ and $Q \in \mathbb G_{n-1}$. Then $|p|<1$.
    For each $j = 1, \dots, n-1$, consider
    \[
    \be_j = \dfrac{s_j - \bar s_{n-j} p }{1-|p|^2}.
    \]
    Then, we have $ \be_j + \bar \be_{n-j} p = s_j$ for each $j = 1, \dots, n-1$. Since $(\be_1,\dots,\be_{n-1})=Q \in \mathbb G_{n-1}$, it follows from Theorem \ref{gn} that $(s_1,\dots,s_{n-1},p) \in \gn$.
\end{proof}

We finally arrive at the main result of this paper, which characterizes a point in $\gn$ in several different ways.

\begin{thm}\label{char gn}
    Let $x=(s_1,\dots,s_{n-1},p) \in \C^n$ and let
\[
Q=\lf \dfrac{s_1 - \bar s_{n-1} p}{1-|p|^2}, \dfrac{s_2 - \bar s_{n-2} p}{1 - |p|^2}, \dots, \dfrac{s_{n-1} - \bar s_1 p}{1 - |p|^2} \rf\,, \text{ when } |p| \neq 1.
\]    
     Then the following are equivalent:
    \begin{enumerate}
        \item[$(1)$] $x \in \gn$;\\
        \item[$(2)$] $Q\in \mathbb G_{n-1}$ and
        ${n \choose j} - s_j z - s_{n-j}w + { n \choose j} pzw \neq 0$, for all $z,w \in \overline\D$ and for all $j = 1, \dots, \left[\frac{n}{2}\right]$ ;\\
        \item[$(3)$] $Q\in \mathbb G_{n-1}$ and for all $j = 1, \dots, \left[\frac{n}{2}\right]$ either
        \[
        \n \Phi_j(.,x)\n_{H^{\infty}} < 1 \; \text{and if } s_js_{n-j}= {n \choose j}^2 p \; \text{ then,}
        \;\text{ in addition, }\; |s_{n-j}|< {n \choose j}
        \]
        or
        \[
        \n \Phi_{n-j}(.,x)\n_{H^{\infty}} < 1 \; \text{and if } s_js_{n-j}= {n \choose j}^2 p \; \text{ then,}
        \;\text{ in addition, }\; |s_j|< {n \choose j} ;
        \]
        \item[$(4)$] $Q\in \mathbb G_{n-1}$ and for all $j = 1, \dots, \left[\frac{n}{2}\right]$ either
        \[
        {n \choose j}\left|s_j - \bar s_{n-j} p\right| + \left|s_j s_{n-j} - {n \choose j}^2 p \right| < {n \choose j}^2 -|s_{n-j}|^2
        \]
        or
        \[
        {n \choose j}\left|s_{n-j} - \bar s_j p\right| + \left|s_j s_{n-j} - {n \choose j}^2 p \right| < {n \choose j}^2 -|s_j|^2 ;
        \]
        \item[$(5)$] $Q\in \mathbb G_{n-1}$ and for all $j = 1, \dots, \left[\frac{n}{2}\right]$ either
        \[
        |s_j|^2 - |s_{n-j}|^2 + {n \choose j}^2|p|^2 + 2{n \choose j}\left|s_{n-j} - \bar s_j p\right| < {n \choose j}^2 \text{ and } \; |s_{n-j}|< {n \choose j}
        \]
        or
        \[
        |s_{n-j}|^2 - |s_j|^2 + {n \choose j}^2|p|^2 + 2{n \choose j}\left|s_j - \bar s_{n-j} p\right| < {n \choose j}^2 \text{ and } \; |s_j|< {n \choose j} ;
        \]
        \item[$(6)$] $Q\in \mathbb G_{n-1}$, $|p|<1$ and
        $|s_j|^2 + |s_{n-j}|^2 - {n \choose j}^2|p|^2 + 2\left| s_js_{n-j} - {n \choose j}^2 p \right| < {n \choose j}^2$ for all $j = 1, \dots, \left[\frac{n}{2}\right]$;\\
        \item[$(7)$] $Q\in \mathbb G_{n-1}$ and $\left| s_{n-j} - \bar s_j p \right| + \left| s_j -\bar s_{n-j} p \right| < {n \choose j} (1 - |p|^2)$ for all $j = 1, \dots, \left[\frac{n}{2}\right]$;\\
        \item[$(8)$] $Q\in \mathbb G_{n-1}$ and there exist $\left[\frac{n}{2}\right]$ number of $2 \times 2$ matrices $B_1,\dots, B_{\left[\frac{n}{2}\right]}$ such that $\n B_j \n < 1$, $s_j = {n \choose j}[B_j]_{11} $, $s_{n-j} = {n \choose j}[B_j]_{22}$ for all $j = 1, \dots, \left[\frac{n}{2}\right]$ and
        \[\det B_1= \det B_2 = \det B_{\left[\frac{n}{2}\right]}= p \; ;\]
        \item[$(9)$] $Q\in \mathbb G_{n-1}$ and there exist $\left[\frac{n}{2}\right]$ number of $2 \times 2$ symmetric matrices $B_1,\dots, B_{\left[\frac{n}{2}\right]}$ such that $\n B_j \n < 1$, $s_j = {n \choose j}[B_j]_{11} $, $s_{n-j} = {n \choose j}[B_j]_{22}$ for all $j = 1, \dots, \left[\frac{n}{2}\right]$ and \[\det B_1= \det B_2 = \det B_{\left[\frac{n}{2}\right]}= p \; .\]
    \end{enumerate}
\end{thm}
\begin{proof}
    The proof follows from Theorem $\ref{char G 3}$ and Theorem $\ref{gn - Gn}$.
\end{proof}

\subsection{Characterizations of a point in $\gamn$}
Recall from the literature (see \cite{edi-zwo}) that the distinguished boundary of $\Gamma_n$ is the smallest closed subset $b\Gamma_n$ of $ \subset \gamn$ such that every complex-valued function on $\Gamma_n$, that is analytic in the interior $\gn$ and continuous on $\gamn$, attains its maximum modulus on $b\gamn$. Evidently, $b\gamn$ is a subset of the topological boundary $\partial \gamn $ of $\gamn$. In \cite{BSR}, the points in the distinguished boundary $b\Gamma_n$ of $\Gamma_n$ are characterized in the following way:
\begin{thm}[\cite{BSR}, Theorem 2.4]\label{thm:DB}
    For $(s_1,\dots,s_{n-1},p)\in\mathbb C^n$ the following are equivalent:
    \begin{enumerate}
        \item $(s_1,\dots,s_{n-1},p)\in b\Gamma_n$\,;
        \item $(s_1,\dots,s_{n-1},p)\in\Gamma_n$
        and $|p|=1$\,;
        \item $|p|=1$, $s_j=\bar s_{n-j}p$ and
        $\left(\dfrac{n-1}{n}s_1,\dfrac{n-2}{n}s_2,\dots, \dfrac{1}{n}s_{n-1}\right)\in\Gamma_{n-1}$\,; \item $|p|=1$
        and there exists $(\be_1,\dots,\be_{n-1})\in b\Gamma_{n-1}$ such that\\
        $ s_j=\be_j+\bar\be_{n-j}p $ , for $j=1,\dots,n-1$.
    \end{enumerate}
\end{thm}

Theorem \ref{thm:DB} provides several independent descriptions of the points in $b\Gamma_n$. Thus, in order to characterize the points in $\Gamma_n$ it suffices if we find characterizations for the points in $\gamn \setminus b\gamn$.

\begin{thm}\label{gamn - Gamn}
    Let $(s_1,\dots,s_{n-1},p) \in \C^n$. Then the following are equivalent:
    \begin{enumerate}
        \item $(s_1,\dots,s_{n-1},p) \in \gamn \setminus b\gamn$;
        \item $(s_1,\dots,s_{n-1},p) \in \Gamma_n$ and
        \[
        \lf \dfrac{s_1 - \bar s_{n-1} p}{1-|p|^2}, \dfrac{s_2 - \bar s_{n-2} p}{1 - |p|^2}, \dots, \dfrac{s_{n-1} - \bar s_1 p}{1 - |p|^2} \rf \in \Gamma_{n-1} \setminus b\Gamma_{n-1}.
        \]
    \end{enumerate}
\end{thm}
\begin{proof}
    Suppose $(s_1,\dots,s_{n-1},p) \in \gamn \setminus b\gamn$. Then $(s_1,\dots,s_{n-1},p) \in \Gamn$ and $|p|\leq1$. By Theorem \ref{thm:DB}, we have $|p|<1$. By Theorem \ref{gn}, there exists $(\be_1,\dots,\be_{n-1}) \in \Gamma_{n-1}$ such that
    \[
    s_j = \be_j + \bar \be_{n-j} p \q \text{for each } j = 1, \dots, n-1.
    \]
    Since $|p|<1$, we have that
    \[
    \be_j = \dfrac{s_j - \bar s_{n-j} p }{1-|p|^2} \q \text{for each } j = 1, \dots, n-1.
    \]
    Thus, $\lf \dfrac{s_1 - \bar s_{n-1} p}{1-|p|^2}, \dfrac{s_2 - \bar s_{n-2} p}{1 - |p|^2}, \dots, \dfrac{s_{n-1} - \bar s_1 p}{1 - |p|^2} \rf \in \Gamma_{n-1}$.
    Since $(s_1,\dots,s_{n-1},p) \notin b\gamn$, by Theorem \ref{thm:DB},
    $\lf \dfrac{s_1 - \bar s_{n-1} p}{1-|p|^2}, \dfrac{s_2 - \bar s_{n-2} p}{1 - |p|^2}, \dots, \dfrac{s_{n-1} - \bar s_1 p}{1 - |p|^2} \rf \in \Gamma_{n-1} \setminus b\Gamma_{n-1}$. Thus $(1) \imp (2)$.

    Conversely, suppose (2) holds. Then $|p|\leq 1$. First suppose $|p|< 1$ and consider
    $
    \be_j = \dfrac{s_j - \bar s_{n-j} p }{1-|p|^2}.
    $
    Then, we have $ \be_j + \bar \be_{n-j} p = s_j$ for each $j = 1, \dots, n-1$. By hypothesis $(\be_1,\dots,\be_{n-1}) \in \Gamma_{n-1}$. Hence by Theorem 3.7 in \cite{costara1}, we have $(s_1,\dots,s_{n-1},p) \in \gamn$. Now suppose $|p|=1$. Since
    $\lf \dfrac{s_1 - \bar s_{n-1} p}{1-|p|^2}, \dfrac{s_2 - \bar s_{n-2} p}{1 - |p|^2}, \dots, \dfrac{s_{n-1} - \bar s_1 p}{1 - |p|^2} \rf \notin  b\Gamma_{n-1}$, by Theorem \ref{thm:DB} we can conclude that $(s_1,\dots,s_{n-1},p) \notin b\gamn$. Hence $(1)$ holds.

\end{proof}

The following theorem, which is another main result of this article, is a consequence of Theorems $\ref{char F 3}$ and $\ref{gamn - Gamn}$ and it provides several independent characterizations of a point in $\gamn\setminus b\gamn$.
\begin{thm}\label{char gamn}
    Let $x=(s_1,\dots,s_{n-1},p) \in \C^n$ and let
\[
Q=\lf \dfrac{s_1 - \bar s_{n-1} p}{1-|p|^2}, \dfrac{s_2 - \bar s_{n-2} p}{1 - |p|^2}, \dots, \dfrac{s_{n-1} - \bar s_1 p}{1 - |p|^2} \rf\,, \text{ when } |p| \neq 1.
\]  
  Then the following are equivalent:
    \begin{enumerate}
        \item[$(1)$] $x \in \gamn \setminus b\gamn$;\\
        \item[$(2)$] $Q\in \mathbb G_{n-1}$ and
        ${n \choose j} - s_j z - s_{n-j}w + { n \choose j} pzw \neq 0$, for all $z,w \in \D$ and $ 1\leq j \leq \left[\frac{n}{2}\right]$ ;\\
        \item[$(3)$] $Q\in \mathbb G_{n-1}$ and for all $j = 1, \dots, \left[\frac{n}{2}\right]$ either
        \[
        \n \Phi_j(.,x)\n_{H^{\infty}} \leq 1 \; \text{and if } s_js_{n-j}= {n \choose j}^2 p \; \text{ then,}
        \;\text{ in addition, }\; |s_{n-j}|\leq {n \choose j}
        \]
        or
        \[
        \n \Phi_{n-j}(.,x)\n_{H^{\infty}} \leq 1 \; \text{and if } s_js_{n-j}= {n \choose j}^2 p \; \text{ then,}
        \;\text{ in addition, }\; |s_j|\leq {n \choose j} ;
        \]
        \item[$(4)$] $Q\in \mathbb G_{n-1}$ and for all $j = 1, \dots, \left[\frac{n}{2}\right]$ either
        \begin{align*}
            &{n \choose j}\left|s_j - \bar s_{n-j} p\right| + \left|s_j s_{n-j} - {n \choose j}^2 p \right| \leq {n \choose j}^2 -|s_{n-j}|^2 \; \; \text{ and if } s_js_{n-j}= {n \choose j}^2 p \\ & \text{ then,}
            \;\text{ in addition, }\; |s_j|\leq {n \choose j}
        \end{align*}
        or
        \begin{align*}
            &{n \choose j}\left|s_{n-j} - \bar s_j p\right| + \left|s_j s_{n-j} - {n \choose j}^2 p \right| \leq {n \choose j}^2 -|s_j|^2 \; \; \text{ and if } s_js_{n-j}= {n \choose j}^2 p \\ & \text{ then,}
            \;\text{ in addition, }\; |s_{n-j}|\leq {n \choose j} ;
        \end{align*}
        \item[$(5)$] $Q\in \mathbb G_{n-1}$ and for all $j = 1, \dots, \left[\frac{n}{2}\right]$ either
        \[
        |s_j|^2 - |s_{n-j}|^2 + {n \choose j}^2|p|^2 + 2{n \choose j}\left|s_{n-j} - \bar s_j p\right| \leq {n \choose j}^2 \text{ and } \; |s_{n-j}|\leq {n \choose j}
        \]
        or
        \[
        |s_{n-j}|^2 - |s_j|^2 + {n \choose j}^2|p|^2 + 2{n \choose j}\left|s_j - \bar s_{n-j} p\right| \leq {n \choose j}^2 \text{ and } \; |s_j|\leq {n \choose j} ;
        \]
        \item[$(6)$] $Q\in \mathbb G_{n-1}$, $|p|< 1$ and
        $|s_j|^2 + |s_{n-j}|^2 - {n \choose j}^2|p|^2 + 2\left| s_js_{n-j} - {n \choose j}^2 p \right| \leq {n \choose j}^2$ for $j = 1, \dots, \left[\frac{n}{2}\right]$;\\
        \item[$(7)$] $Q\in \mathbb G_{n-1}$ and
        $\left| s_{n-j} - \bar s_j p \right| + \left| s_j -\bar s_{n-j} p \right| \leq {n \choose j} (1 - |p|^2)$ for $j = 1, \dots, \left[\frac{n}{2}\right]$ ;\\
        \item[$(8)$] $Q\in \mathbb G_{n-1}$ and there exist $\left[\frac{n}{2}\right]$ number of $2 \times 2$ matrices $B_1,\dots, B_{\left[\frac{n}{2}\right]}$ such that $\n B_j \n \leq 1$, $s_j = {n \choose j}[B_j]_{11} $, $s_{n-j} = {n \choose j}[B_j]_{22}$ for all $j = 1, \dots, \left[\frac{n}{2}\right]$ and
        $ \det B_1= \det B_2 = \det B_{\left[\frac{n}{2}\right]}= p$ ;\\
        \item[$(9)$] $Q\in \mathbb G_{n-1}$ and there exist $\left[\frac{n}{2}\right]$ number of $2 \times 2$ symmetric matrices $B_1,\dots, B_{\left[\frac{n}{2}\right]}$ such that $\n B_j \n \leq 1$, $s_j = {n \choose j}[B_j]_{11} $, $s_{n-j} = {n \choose j}[B_j]_{22}$ for all $j = 1, \dots, \left[\frac{n}{2}\right]$ and $ \det B_1= \det B_2 = \det B_{\left[\frac{n}{2}\right]}= p $.
    \end{enumerate}
\end{thm}

\begin{proof}
The proof is similar to the proof of Theorem \ref{char gn}.
\end{proof}

\section{Some more geometric properties of $\Gn$ and $\Gamn$}

\vspace{0.3cm}

\noindent We further study the geometry of the domain $\Gn$ and its closure $\Gamn$ and obtain a few important features. We prove that $\Gamn$ is not convex but is polynomially convex. Since the closed symmetrized bidisc $\Gamma_2$ is not convex (see Section 2 of \cite{ay-jfa} and \cite{edi-zwo}), it follows that $\widetilde{\Gamma_2}$ is not convex because $\Gamma_2=\widetilde{\Gamma_2}$ (see Remark \ref{rem:23}). Below we show that $\Gamn$ is not convex for any $n\geq 2$. Also, we show that both $\Gn$ and $\Gamn$ are starlike about the origin but not circled. It follows from here that $\Gn$ is simply connected. For the convenience of a reader, we begin with a few definitions from the literature.
\begin{defn}
	A compact set $K \subset \mathbb{C}^n$ is called polynomially convex if for any $y \in \mathbb{C}^n \setminus K $, there exists a polynomial in $n$ variables, say $P$, such that
	$$|P(y)| > \underset{z \in K}{\sup}|P(z)| = \lVert P \lVert_{\infty,K}. $$ 
\end{defn}

\begin{defn}
	A set $S \subset \mathbb{C}^n$ is said to be starlike if there exists a $z_0 \in S$ such that for all $z \in S$ the line segment joining $z$ and $z_0$, lies in $S$.
\end{defn}

\begin{defn}
	A set $S \subset \mathbb{C}^n$ is said to be circular if $z\in S$ implies that $(e^{\iota\theta}z_1,e^{\iota\theta}z_2,...,e^{\iota\theta}z_n) \in S$ for all $0\leq \theta < 2\pi$.	
\end{defn}
Consider the points $a=(n,0,\dots,0,ni,i) \in \C^n$ and $b=(-ni,0 \dots, 0,n,-i) \in \C^n$. Then $a , b \in \Gamn$, because, both the points $a$ and $b$ satisfy condition $(7)$ of Theorem $\ref{char F 3}$. The mid point of the line joining $a$ and $b$ is
    \[
    c = \frac{1}{2}(a+b) = \lf \dfrac{n-ni}{2},0,\dots,0, \dfrac{ni + n}{2},0 \rf \in \C^n.
    \]
    The point $c \notin \Gamn$. This is because the inequality in condition $(7)$ of Theorem $\ref{char F 3}$, corresponding to the point $c$ and $j=1$, does not hold. Indeed, after substituting the values in that particular inequality we get the following
    \[
    \left|\dfrac{ni + n}{2} \right| + \left|\dfrac{n-ni}{2} \right| = \sqrt{2}n \nleq {n \choose 1}.
    \]
    Hence, the line segment joining $a$ and $b$ does not entirely lie within $\Gamn$ and consequently $\Gamn$ is not convex.\\
    
We now show that $\Gamn$ is polynomially convex. The techniques that we use to prove this are similar to that in \cite{awy}, where the authors established that the tetrablock is polynomially convex.

\begin{thm}
    $\Gamn$ is polynomially convex.
\end{thm}
\begin{proof}
    Consider a point $ y= (y_1,\dots,y_{n-1},q) \in \Cn \setminus \Gamn$. To show $\Gamn$ is polynomially convex, it is enough to find a polynomial $f$ in $n$ variables, such that $|f| \leq 1$ on $\Gamn$ and $|f(y)|>1 \, $. Let $|y_j| > \nj$ for some $j = 1, \dots, \left[\frac{n}{2}\right]$. Then $f\lf (x_1,\dots,x_n) \rf = \dfrac{x_j}{\nj}$  has the required property, since for any $x \in \Gamn$, $|x_j|\leq \nj$ for all $j = 1, \dots, \left[\frac{n}{2}\right]$. If $|q|>1$ then take $f\lf (x_1,\dots,x_n) \rf=x_n \, $. Thus assume $y$ to be such that $|y_j| \leq \nj$ for all $j =1, \dots, n-1$; and $|q|< 1$. Since $y \notin \Gamn$, by condition $(3)$ of Theorem $\ref{char G 3}$ there exist some $j \in \left\{1, \dots, \left[\frac{n}{2}\right] \right\}$ such that $\n \Phi_j(.,y)\n >1$. That is, there exist $z \in \D$ and some $j \in \left\{1, \dots, \left[\frac{n}{2}\right] \right\}$ such that $\left| \Phi_j(z,y) \right| >1$. Fix that $j$ and $z$. Note the following cases :
    \item[case 1]: If $y_jy_{n-j} = \nj^2q$, then $|\Phi_j(z,y)| = \dfrac{|y_j|}{\nj}$. Hence $|y_j| > \nj$ and then the function $f \lf (x_1,\dots,x_n) \rf =\dfrac{x_j}{\nj}$ is the required function.
    \item[case 2] :  If $y_jy_{n-j} \neq \nj^2q$, then $\Phi_j(z,y) = \dfrac{{n \choose j}qz-y_j}{y_{n-j}z-{n \choose j}}$. For $m \in \mathbb{N}$, define a polynomial $f_m$ in $n$ variables as follows
    \[
    f_m\lf x_1, \dots , x_n \rf = \lf\dfrac{x_j}{\nj} - x_nz \rf\lf 1+ \frac{x_{n-j}z}{\nj} + \lf \frac{x_{n-j}z}{\nj}\rf^2 + \cdots + \lf\frac{x_{n-j}z}{\nj}\rf^{m} \rf.
    \]
    Let \[\mathcal{W} = \ls (w_1,\dots,w_n) \in \Cn : |w_j| \leq \nj, j=1,\dots.n  \rs.\]
    Then $\Gamn \subset \mathcal{W}$ and also $y \in \mathcal W$. For $x \in \mathcal W$, we have $\left| \dfrac{x_{n-j}z}{\nj} \right| \leq |z| < 1$ and hence $\Phi_j(z,x)$ can be written in the following form
    \begin{align*}
        \Phi_j(z,x) &= \dfrac{\dfrac{x_j}{\nj} - x_n z}{1-\dfrac{x_{n-j}}{\nj}z}\\
        &= \lf\dfrac{x_j}{\nj} - x_nz \rf\lf 1+ \frac{x_{n-j}z}{\nj} + \lf \frac{x_{n-j}z}{\nj}\rf^2 + \cdots  \rf .
    \end{align*}
    For $x \in \mathcal W$, we also have
    \begin{align*}
        & \left|\dfrac{x_j}{\nj} - x_n z \right| \leq \left| \dfrac{x_j}{\nj} \right| +   |x_n|| z| \leq 2  \qq (\text{as } z \in \D )\\
        \text{ and } \q & \left|1-\dfrac{x_{n-j}z}{\nj}\right| \geq 1-\left|\dfrac{x_{n-j}z}{\nj}\right| \geq 1-|z|.
    \end{align*}
    Then for any $x \in \mathcal W$ and for any $m \in \mathbb N$, we have
    \begingroup
    \allowdisplaybreaks
    \begin{align*}
        |f_m(x) - \Phi_j(z,x)| &= \left| \dfrac{x_j}{\nj} - x_n z \right| \left| \lf \dfrac{x_{n-j}z}{\nj} \rf^{m+1} + \lf \dfrac{x_{n-j}z}{\nj}\rf^{m+2} + \cdots \right|\\
        & \leq 2 \left| \lf \dfrac{x_{n-j}z}{\nj} \rf^{m+1} \right| \left| 1+ \dfrac{x_{n-j}z}{\nj} + \lf \dfrac{x_{n-j}z}{\nj} \rf^2 + \cdots  \right|\\
        & = 2\dfrac{\left| \lf\dfrac{x_{n-j}z}{\nj} \rf^{m+1} \right|}{\left|1-\dfrac{x_{n-j}z}{\nj}\right|}\\
        & \leq \dfrac{2|z|^{m+1}}{1- |z|}.
    \end{align*}
    \endgroup
    Let $0< \epsilon < \dfrac{1}{3}\lf|\Phi_j(z,y)| - 1\rf$. Note that $z \in \D$ implies $|z|^{m+1}\rightarrow 0$ as $m \rightarrow \infty$. Hence there exists $k \in \mathbb N$, large enough, such that $\dfrac{2|z|^{k+1}}{1- |z|} < \epsilon$. Since
    \[
    |f_m(x) - \Phi_j(z,x)| \leq \dfrac{2|z|^{m+1}}{1- |z|} \q  \text{for all }  x \in \mathcal W \text{ and for any } m \in \mathbb N,
    \]
    we have
    \[
    |f_k(x) - \Phi_j(z,x)| < \epsilon \q \text{for all }\, x \in \mathcal W.
    \]
    Again since $\Gamn \subset \mathcal W$ and $|\Phi_j(z,x)|\leq 1$ for all $x \in \Gamn$, we have
    \[
    |f_k| \leq |f_k - \Phi_j(z,x)| + |\Phi_j(z,x)| \leq 1+ \epsilon \q \text{for all } x \in \Gamn.
    \]
    Note that, $|\Phi_j(z,y)| > 1+ 3\epsilon$. Since $y \in \mathcal W$, we have $|\Phi_j(z,y) - f_k(y)| < \epsilon$, that is
    \[
    |\Phi_j(z,y)| - |f_k(y)| \leq |\Phi_j(z,y) - f_k(y)| < \epsilon .
    \]
    Therefore,
    \[
    |f_k(y)| > |\Phi_j(z,y)| - \epsilon > 1 + 2\epsilon.
    \]
    Now take $f = (1+\epsilon)^{-1}f_k$. Then $|f(x)| \leq 1$ for all $x\in \Gamn$, and
    \[
    |f(y)|= \dfrac{|f_k(y)|}{1+\epsilon} >\dfrac{1 + 2\epsilon}{1 + \epsilon} >1.
    \]
    Hence $f$ is the required polynomial.
\end{proof}

% %  -------- Theorem Starlike ---------------- % %
\begin{thm}
    $\Gn$ and $\Gamn$ are both starlike about $(0,\dots,0)$ but not circular. Hence $\Gn$ is simply connected.
\end{thm}
\begin{proof}
    Let $y = (y_1,\dots,y_{n-1},q) \in \Gamn$ and let $0 \leq r < 1$. To prove that $\Gn$ and $\Gamn$ are starlike about $(0,0,0)$, it is enough to show that $ry \in \Gn$ for all such $y$ and $r $.
    First note that, for $r >0$, $z \in \C$ and $j=1, \dots, n-1$ we have
    \begingroup
    \allowdisplaybreaks
    \begin{align}{\label{starlike}}
        \nonumber & \left|\nj - rzy_{n-j}\right|^2 - \left| ry_j - \nj rzq \right|^2 \\ & \nonumber =  \nj^2 - 2\nj r\text{Re}(zy_{n-j}) + r^2|zy_{n-j}|^2 - \left| ry_j - \nj rzq \right|^2 \\
        \nonumber & = r^2\left\{\left|\nj - zy_{n-j}\right|^2 - \left|y_j - \nj zq\right|^2\right\}\\
        &\q + (1-r)\left(\nj^2 + \nj^2 r -2\nj r\text{Re}(zy_{n-j})\right) \, .
    \end{align}
    \endgroup
    Since $y \in \Gamn$, by condition $(3)$ of Theorem $\ref{char F 3}$, we have $\n \Phi_j(.,y)\n \leq 1$ for all $j = 1, \dots, \left[\frac{n}{2}\right]$. Hence, for each $j = 1, \dots, \left[\frac{n}{2}\right]$
    \begingroup
    \allowdisplaybreaks
    \begin{align*}
      \Phi_j(.,y)\n \leq 1 & \Lra  \left| \Phi_j(z,y) \right| \leq 1 \q \text{for all } z\in \overline\D \\
        \Lra & \left|\nj - zy_{n-j}\right|^2 - \left|y_j - \nj zq\right|^2 \geq 0 \, \q \text{ for any } z \in \overline{\mathbb{D}}. 
  \end{align*}
    \endgroup
    Again for any $z \in \overline{\D}$ and for all $j = 1, \dots, \left[\frac{n}{2}\right]$, we have
   \[
   \nj^2 + \nj^2r -2\nj r\text{Re}(zy_{n-j})= \left(\nj^2 - r|zy_{n-j}|^2\right) + r\left|\nj - zy_{n-j} \right|^2 > 0.
    \]
    This is because $r|zy_{n-j}|^2 < |zy_{n-j}|^2 \leq |y_{n-j}|^2 \leq \nj^2$ (as $y \in \Gamn$). Hence from equation $\eqref{starlike}$, we have
    \[ \left|\nj - rzy_{n-j}\right|^2 - \left|ry_j - \nj rzq \right|^2 > 0 \, ,\]
    for any $r\in[0,1)$, $z \in \overline{\D}$ and for all $j = 1, \dots, \left[\frac{n}{2}\right]$.
    Therefore,
    \[
    \left|\Phi_j(z,ry) \right| = \dfrac{\left|ry_j - \nj rzq \right|}{\left|\nj - rzy_{n-j}\right|} <1 \q \text{for all } j = 1, \dots, \left[\frac{n}{2}\right],
    \]
    whenever $y \in \Gamn$, $z \in \overline{\mathbb{D}}$ and
    $0 \leq r < 1 $. Thus, $\n\Phi_j(.,ry) \n <1 $ for all $j = 1, \dots, \left[\frac{n}{2}\right]$, whenever $y \in \Gamn$ and $0 \leq r < 1$.
    Therefore, by Theorem \ref{char G 3}, $ry \in \Gn$ whenever $y \in \Gamn$ and $0 \leq r < 1 $.
    Hence $\Gn$ and $\Gamn$ are starlike about $ (0,\dots,0) $.\\

    Next we show that $\Gamn $ is not a circular, which implies that $\Gn$ is not a circular domain either. Let $y= (y_1,\dots,y_{n-1},q) \in \Cn$ be such that $y_j = \nj $ for $j = 1, \dots, \left[\frac{n}{2}\right]$. and $q=1$, that is,
    \[ y =\lf {n \choose 1},\dots ,{n \choose {n-1}},1 \rf .\]
    To prove $\Gamn $ is not a circular, we show that $y \in \Gamn$ but $iy \notin \Gamn$. Note that, for the above mentioned point $y$, we have $y_jy_{n-j} = \nj^2q$, and
    \[
    \nj \left| y_j - \bar y_{n-j}q \right| + \left| y_jy_{n-j} - \nj^2 q \right| = 0 ,\q\text{ for each }\, j = 1, \dots, \left[\frac{n}{2}\right]
    \]
    and \[\nj^2 -|y_{n-j}|^2 = 0, \q\text{ for each }\, j = 1, \dots, \left[\frac{n}{2}\right]. \]
    Therefore, by condition $(4)$ of Theorem $\ref{char G 3}$, $y \in \Gamn$. Now consider the point
    \[\tilde y= iy =\lf i{n \choose 1},\dots ,i{n \choose {n-1}},i \rf .\]
    Then for each $j = 1, \dots, \left[\frac{n}{2}\right]$, we have
    $\nj^2 -|\tilde y_{n-j}|^2 = 0$. Again for each $j = 1, \dots, \left[\frac{n}{2}\right]$,
    \[
    \nj \left|\tilde y_j - \bar{\tilde y}_{n-j}q \right| + \left| \tilde y_j \tilde y_{n-j} - \nj^2 q \right|= 2\nj^2|1+i| >0 .
    \]
    Thus $\tilde y$ does not satisfy condition $(4)$ of Theorem $\ref{char G 3}$ and hence $\tilde y = iy \notin \Gamn$. Thus $\Gamn$ is not circular. Thus it follows from here that $\Gn$ is simply connected.

\end{proof}

\noindent \textbf{Acknowledgement.} We are thankful to the referee for making fruitful suggestions which help in improving the exposition of the article.

\end{document}